\documentclass[a4paper,12pt]{amsart}
\usepackage{amsmath}
\usepackage{amsmath}
\usepackage{amssymb}
\usepackage{amscd}
\usepackage{mathrsfs}
\usepackage[all]{xy}
\usepackage[dvipdfm]{graphicx}
\def\mathcal{\mathscr}
\everymath{\displaystyle}
\newtheorem{thm}{Theorem}[section]
\newtheorem{lem}[thm]{Lemma}
\newtheorem{cor}[thm]{Corollary}
\newtheorem{prop}[thm]{Proposition}

\theoremstyle{definition}

\newtheorem{rem}[thm]{\rm{Remark}}

\newtheorem{defn}[thm]{\rm{Definition}}

\newcommand{\mf}[1]{{\mathfrak{#1}}}

\newcommand{\mca}[1]{{\mathcal{#1}}}
\def\ad{\text{\rm ad}}

\def\Z{{\mathbb Z}}
\def\C{{\mathbb C}}

\def\R{{\mathbb R}}

\def\Cok{\text{\rm Coker}}
\def\Crit{\text{\rm Crit}}
\def\Cyl{\text{\rm Cyl}}
\def\dR{\text{\rm dR}}
\def\WFC{\text{\rm WFC}}
\def\WFH{\text{\rm WFH}}

\def\Im{\text{\rm Im}\,}
\def\ind{\text{\rm ind}}
\def\inn{\text{\rm in}}
\def\interior{\text{\rm int}}

\def\reg{\text{\rm reg}}

\def\RS{\text{\rm RS}}
\def\supp{\text{\rm supp}}
\def\st{\text{\rm st}}
\begin{document}
\pagestyle{plain}
\thispagestyle{plain}

\title[Handle attaching in Wrapped Floer Homology and brake orbits in classical Hamiltonian systems]
{Handle attaching in Wrapped Floer Homology and brake orbits in classical Hamiltonian systems}

\author[Kei Irie]{Kei Irie}
\address{Research Institute for Mathematical Sciences, Kyoto University, 
Kyoto 606-8502, Japan}
\email{iriek@kurims.kyoto-u.ac.jp}

%\date{\today}

\begin{abstract}
The objective of this note is to prove an existence result for brake orbits in classical Hamiltonian systems
(which was first proved by S.V.Bolotin) by using Floer theory.
To this end, we compute an open string analogue of symplectic homology (so called wrapped Floer homology)
of some domains in cotangent bundles, which appear naturally in the study of classical Hamiltonian systems.
The main part of the computations is to show invariance of wrapped Floer homology under certain handle attaching 
to domains. 
\end{abstract}

\maketitle

\section{Introduction}
First we recall the definition of classical Hamiltonian systems.
Let $N$ be a $n$-dimensional manifold.
Then, $T^*N$ carries a symplectic form $\omega_N:=\sum_{1 \le i \le n} dp_i \wedge dq_i$ 
where $(q_1,\ldots,q_n)$ is a local coordinate in $N$, and 
$(p_1,\ldots,p_n)$ is the associated coordinate on fibers.

Assume that $N$ carries a Riemannian metric.
Then, for $V \in C^\infty(T^*N)$, we define $H_V \in C^\infty(T^*N)$ by
$H_V(q,p)=V(q)+|p|^2/2$.
A pair of symplectic manifold $(T^*N,\omega_N)$ and $H_V \in C^\infty(T^*N)$ is called \textit{classical Hamiltonian system}.
Its \textit{Hamiltonian vector field} $X_{H_V}$ is defined by $i_{X_{H_V}}\omega_N =-dH_V$.
As is well-known, $X_{H_V}$ describes free motion of a particle on $N$ under potential energy given by $V$.

The following theorem is first proved by S.V.Bolotin \cite{Bolotin}.

\begin{thm}\label{thm:periodicorbits}
Let $N$ be a Riemannian manifold, and $V \in C^\infty(N)$.
If $S_h:=H_V^{-1}(h)$ is a compact and regular hypersurface in $T^*N$, 
then there exists a closed orbit of $X_{H_V}$ on $S_h$.
\end{thm}

When $S_h \cap N = \emptyset$, Theorem \ref{thm:periodicorbits} is easily obtained by the existence of 
closed geodisics on compact Riemannian manifolds, using Maupertuis-Jacobi principle.
So difficulty arises when $S_h \cap N \ne \emptyset$.
In this case, Theorem \ref{thm:periodicorbits} is obtained by the following result (\cite{Bolotin}):

\begin{thm}\label{thm:brakeorbits}
Let $N$ and $V$ are as in Theorem \ref{thm:periodicorbits}.
If $S_h \cap N \ne \emptyset$, there exists a non-trivial orbit of $X_{H_V}$ on $S_h$
, which starts from and ends at
$S_h \cap N$.
\end{thm}

Define $I:T^*N \to T^*N$ by $I(q,p)=(q,-p)$.
If $x:[0,l] \to S_h$ satisfies $\dot{x}=X_{H_V}(x)$ and $x(0), x(l) \in N$, then
$\overline{x}: [0,2l] \to S_h$ defined by
\[
\overline{x}(t) = \begin{cases}
             x(t) &(0 \le t \le l) \\
             I\bigl(x(2l-t)\bigr) &(l \le t \le 2l)
             \end{cases}
\]
is a closed orbit of $X_{H_V}$ (closed orbits of $X_{H_V}$ obtained in this way are so-called \textit{brake orbits}).
Hence Theorem \ref{thm:brakeorbits} implies Theorem \ref{thm:periodicorbits}.

In this paper, we deduce Theorem \ref{thm:brakeorbits} from computations of 
certain Floer-theoric invariant.
The invariant we use is an open string analogue of symplectic homology, and 
often called \textit{wrapped Floer homology}.
Foundations of wrapped Floer homology can be found in \cite{AbouzaidSeidel}
(they also construct an $A^\infty$-algebra structure on the chain complex underlying the homology).
Roughly speaking, wrapped Floer homology is defined for $(M,\omega,L)$, where $(M,\omega)$ is a compact 
symplectic manifold with contact type boundary, and $L$ is a Lagrangian of $(M,\omega)$ (in a precise sense, we need more 
data and additional conditions. see section 2 for details).
Let us denote the wrapped Floer homology for $(M,\omega,L)$ by $\WFH_*(M,\omega,L)$.

We explain our main theorem briefly.
Let $N$ be a Riemannian manifold, and $V \in C^\infty(N)$.
Assume that $S_h=H_V^{-1}(h)$ is compact.
Then, setting $D_h:=H_V^{-1}\bigl((-\infty,h]\bigr)$, $(D_h,\omega_N)$ is a compact symplectic manifold with 
contact type boundary, and we can define wrapped Floer homology for $(D_h,\omega_N,D_h \cap N)$ (for details, 
see section 4). Our main theorem is Theorem \ref{mainthm}, which asserts that
if $S_h \cap N \ne \emptyset$ and $D_h$ is connected, then $\WFH_*(D_h,\omega_N,D_h \cap N)=0$.

Combined with basic results of wrapped Floer homology, Theorem \ref{mainthm} implies 
Theorem \ref{thm:brakeorbits} (Details are explained in section 4).
Theorem \ref{mainthm} is proved as follows.
By "deformation invariance" of wrapped Floer homology (Proposition \ref{prop:isotopyinvarianceofWFH}), 
it is easy to show that $\WFH_*(D_h,\omega_N,D_h \cap N)$ depends only on diffeomorphism type of $D_h \cap N$.
When $D_h \cap N$ is diffeomorphic to the disk, 
it is easy to check that $\WFH_*(D_h,\omega_N,D_h \cap N)=0$.
Hence all we have to show is the invariance of 
$\WFH_*(D_h,\omega_N,D_h \cap N)$ under surgery on $D_h \cap N$ by attaching handles
(Lemma \ref{lem:handleattaching}).
This is proved by arguments which are similar to Cieliebak's arguments in \cite{Cieliebak}, 
where he proves the invariance of symplectic homology under subcritical handle attaching.

We explain the structure of this paper.
In section 2, we recall basics of wrapped Floer homology.
We treat somewhat broader class of Hamiltonians 
than usually considered in Floer theory for manifolds with boundary, 
because this is needed to carry out arguments in section 5.
For this reason, establishing $C^0$ estimate for Floer trajectories becomes harder than usual.
The precise statement of the $C^0$ estimate is stated in section 2 (Theorem \ref{thm:c0estimate}), and 
proved in section 3.
The proof given in section 3 is based on \cite{FloerHofer}.
In section 4, we explain basics of classical Hamiltonian systems, and state the main theorem (Theorem \ref{mainthm}).
We also reduce Theorem \ref{mainthm} to Lemma \ref{lem:handleattaching} in section 4.
Lemma \ref{lem:handleattaching} is proved in section 5.

Acknowledgements. I would like to appreciate professor Kenji Fukaya, for reading manuscripts of this note 
and giving precious suggestions.

\section{Wrapped Floer homology}

In this section, we recall basics of wrapped Floer homology, which we will use in the following of this paper.

\subsection{Liouville quadruple}

First we define the notion of \textit{Liouville quadruples}, for which we define wrapped Floer homology.

\begin{defn}
Let $(M,\omega)$ be a $2n$ dimensional compact symplectic manifold, 
$X \in \mca{X}(M)$, and $L$ be a Lagrangian of $M$.
A \textit{Liouville quadruple} is a quadruple $(M,\omega,X,L)$ with 
following properties:
\begin{enumerate}
\item $L_X\omega=\omega$.
\item $X$ points strictly outwards on $\partial M$.
\item $X_q \in T_qL$ for any $q \in L$.
\item $L$ is transverse to $\partial M$, and $\partial L = L \cap \partial M$.
\end{enumerate}
\end{defn} 

For a Liouville quadruple $(M,\omega,X,L)$, let 
$\lambda:=i_X\omega$.
Then, $\lambda|_L=0$.
$(\partial M, \lambda)$ is a contact manifold, and 
$\partial L$ is a Legendrean of $(\partial M, \lambda)$.
Recall that the \textit{Reeb vector field} $R$  on $(\partial M, \lambda)$ is characterized by $i_R\omega=0$, $\lambda(R)=1$.
Let $\mca{C}(\partial M, \lambda, \partial L)$ be the set of all \textit{Reeb chords} of $\partial L$ in $(\partial M, \lambda)$, i.e.
\[
\mca{C}(\partial M, \lambda, \partial L):
=\bigl\{x \colon [0,l] \to \partial M \bigm{|}l>0, x(0),x(l) \in \partial L, \, \dot{x}=R(x) \bigr\}.
\]
For $x \in \mca{C}(\partial M, \lambda, \partial L)$, let $\mca{A}(x):=\int_0^l x^*\lambda$.
Define the \textit{action spectrum} of $\partial L$
\[
\mca{A}(\partial M, \lambda, \partial L):=\bigl\{\mca{A}(x) \bigm{|} x \in \mca{C}(\partial M, \lambda, \partial L) \bigr\}.
\]
It is easy to verify that $\inf \mca{A}(\partial M, \lambda, \partial L)>0$.

Let $\hat{M}:=M \cup \partial M \times [1,\infty)$.
We extend $X \in \mca{X}(M)$ to $\hat{X} \in \mca{X}(\hat{M})$ by
$\hat{X} = \rho \partial_\rho$ on $\partial M \times [1,\infty)$, where
$\rho$ stands for coordinate on $[1,\infty)$.
Moreover, we extend $\lambda$ to $\hat{\lambda}$ by $\hat{\lambda}:=\rho\lambda$ on $\partial M \times [1,\infty)$, 
and $\omega$ to $\hat{\omega}:=d\hat{\lambda}$.
Then, $\hat{L}:= L \cup \partial L \times [1,\infty)$ is a Lagrangian of $(\hat{M},\hat{\omega})$.
We call $(\hat{M},\hat{\omega},\hat{X},\hat{L})$ the \textit{completion} of
$(M,\omega,X,L)$.

Define $\Phi \colon \partial M \times (0,\infty) \to \hat{M}$ by 
\[
\Phi(z,1) = z, \qquad \partial_\rho\Phi(z,\rho) = \rho^{-1}\hat{X}\bigl(\Phi(z,\rho)\bigr).
\]
Then, $\Phi^*\hat{\lambda}=\rho\lambda$.
We call $\Im (\Phi)$ the \textit{cylindrical part} of $\hat{M}$, and denote it by $\Cyl(\hat{M})$.
We often identify $\Cyl(\hat{M})$ with $\partial M \times (0,\infty)$ via $\Phi$.
For any $\rho \in (0,\infty)$, we define $M(\rho)$ to be the domain in $\hat{M}$, which is 
 bounded by the hypersurface $\partial M \times \{\rho\}$. i.e.
\[
M(\rho):= \begin{cases}
       M \cup \partial M \times (1,\rho] &(\rho \ge 1) \\
       M \setminus \partial M \times (\rho,1] &(\rho < 1).
       \end{cases}
\]

\subsection{Chords and indexes}

For $H \in C^\infty(\hat{M})$, let 
\[
\mca{C}(H)
:=\bigl\{ x \colon [0,1] \to \hat{M} \bigm{|} x(0),x(1) \in \hat{L},\,\dot{x}=X_H(x) \bigr\},
\]
where $X_H$ is the \textit{Hamiltonian vector field} of $H$, defined by 
$dH=-i_{X_H}\hat{\omega}$.

For $x \in \mca{C}(H)$ and $0 \le t \le 1$, let $\Phi_t \colon T_{x(0)}\hat{M} \to T_{x(t)}\hat{M}$ be the 
Poincare map of the flow generated by $X_H$.
$x \in \mca{C}(H)$ is called \textit{nondegenerate} if $\Phi_1 \colon T_{x(0)}\hat{M} \to T_{x(1)}\hat{M}$ satisfies
$\Phi_1(T_{x(0)}\hat{L}) \cap T_{x(1)}\hat{L} = 0$.

For nondegenerate $x \in \mca{C}(H)$, we define its index $\ind(x)$.
In the following of this paper, we assume that any Liouville quadruple $(M,\omega,X,L)$ satisfies
\[
\pi_1(M,L)=\pi_2(M,L)=0.
\]
This is quite strong assumption, but it is enough to consider this case for our objective.

Consider ${\R}^{2n}$ with coordinate $(q_1,\ldots,q_n,p_1,\ldots,p_n)$ and
the \textit{standard symplectic form} $\omega_\st:=\sum_{1 \le i \le n} dp_i \wedge dq_i$.
Let $\mca{L}(n)$ be the space of Lagrangian subspaces of $({\R}^{2n},\omega_\st)$.
Note that $\{p=0\} \in \mca{L}(n)$.

Let $x \in \mca{C}(H)$ and assume that $x$ is nondegenerate.
Let $D^+\colon=\{ z \in \C \mid |z| \le 1, \Im z \ge 0\}$ and
take $\overline{x} \colon D^+ \to \hat{M}$ such that
$\overline{x}(e^{i\pi\theta})=x(\theta)\,(0 \le \theta \le 1)$ and
$\overline{x}(D^+ \cap \R) \subset \hat{L}$
(such $\overline{x}$ exists since $\pi_1(\hat{M},\hat{L})=0$).
Take arbitrary isomorphism of vector bundles
$F: \overline{x}^*T\hat{M} \to (\R^{2n},\omega_\st) \times D^+$ over $D^+$,
such that $F_z: T_{\overline{x}(z)}\hat{M} \to \R^{2n}$ preserves symplectic form for any $z \in D^+$, and 
$F_z(T_{\overline{x}(z)}\hat{L})=\{p=0\}$ for any $z \in D^+\cap\R$.
Define $\Lambda: [0,1] \to \mca{L}(n)$ by 
$\Lambda (\theta):= F_{e^{i\pi\theta}}\bigl(\Phi_{\theta}(T_{x(0)}\hat{L})\bigr)$, and let
\[
\ind(x):= \frac{n}{2} + \mu_\RS\bigl(\Lambda,\{p=0\}\bigr),
\]
where $\mu_\RS$ is the Robbin-Salamon index introduced in \cite{Robbin-Salamon}.
Note that this definition is independent of the choice of $\overline{x}$ since $\pi_2(\hat{M},\hat{L})=0$.

\subsection{Hamiltonians}
Let $K$ be a compact set in $\hat{M}$ which contains $M$.
Then, $H \in C^\infty(\hat{M})$ is of \textit{contact type} on $\hat{M} \setminus K$, if and only if there exists a
smooth positive function $a$ on $\partial M$ and $b \in \R$ such that 
\[
(z,\rho) \in \hat{M}\setminus K \implies H(z,\rho)=a(z)\rho+b.
\]
$a$ and $b$ are uniquely determined by $H$, and denoted by $a_H$, $b_H$.
The set of all $H \in C^\infty(\hat{M})$ which are of contact type on $\hat{M} \setminus K$ 
is denoted by $\mca{H}_K(\hat{M})$.
$H \in \mca{H}_K(\hat{M})$ is called \textit{admissible} if $1 \notin \mca{A}(\partial M, a_H^{-1}\lambda, \partial L)$
and all elements of $\mca{C}(H)$ are nondegenerate.
The set of all admissible elements of
$\mca{H}_K(\hat{M})$ is denoted by $\mca{H}_{K,\ad}(\hat{M})$.
Let $\mca{H}(\hat{M}):=\bigcup_K \mca{H}_K(\hat{M})$ and $\mca{H}_\ad(\hat{M}):=\bigcup_K \mca{H}_{K,\ad}(\hat{M})$, 
%Let $\mca{H}(\hat{M})$ and $\mca{H}_{\ad}(\hat{M})$ be the union of $\mca{H}_K(\hat{M})$ and $\mca{H}_{K,\ad}(\hat{M})$
where $K$ runs over all compact sets in $\hat{M}$ which contain $M$.
It is easy to verify that if $H \in \mca{H}_\ad(\hat{M})$, then 
$\mca{C}(H)$ is a finite set.

Let $H, H' \in \mca{H}_{\ad}(\hat{M})$. 
$(H^s)_{s \in \R}$, a smooth family of elements of $\mca{H}(\hat{M})$, is called
\textit{monotone homotopy} from $H$ to $H'$, if it satisfies following conditions:
\begin{enumerate}
\item There exists a compact set $K$ such that $H^s \in \mca{H}_K(M)$ for any $s$.
\item There exists $s_0>0$ such that:
\begin{enumerate}
\item $H^s=\begin{cases}
                                      H &(s \le -s_0)\\
                                      H'&(s \ge s_0)\\
                                      \end{cases}$.
\item For any $s \in (-s_0,s_0)$, $\partial_s a_{H^s}(z)>0$ for any $z \in \partial M$.
\end{enumerate}
\end{enumerate}

\subsection{Almost complex structures}

Let $J$ be an almost complex structure on $\hat{M}$.
$J$ is \textit{compatible with} $\hat{\omega}$ if and only if
\[
\langle\,\cdot\,,\,\cdot\,\rangle_J \colon T\hat{M} \times T\hat{M} \to \R; \quad (v,w) \mapsto \hat{\omega}(v,Jw)
\]
is a Riemannian metric on $\hat{M}$.
We denote the set of almost complex structures on $\hat{M}$ which are compatible with $\hat{\omega}$
by $\mca{J}(\hat{M},\hat{\omega})$. We often abbreviate it as $\mca{J}(\hat{M})$.

For smooth positive function $a$ on $\partial M$, define diffeomorphism
\[
\Phi_a : \partial M \times (0,\infty) \to \Cyl(\hat{M});\quad (z,\rho) \mapsto \bigl(z, a(z)^{-1} \rho \bigr).
\]

Let $\lambda^a:=a^{-1}\lambda \in \Omega^1(\partial M)$.
Then, $(\Phi_a)^*(\hat{\lambda})=\rho\lambda^a$.
Let $\xi^a$ and $R^a$ be the contact distribution and the Reeb flow on $(\partial M, \lambda^a)$.

For $v \in T(\partial M)$, let 
\[
\overline{v}:=(v,0) \in T(\partial M) \oplus \R \partial_\rho = T\bigl(\partial M \times (0,\infty)\bigr).
\]
There is a natural decomposition 
\[
T\bigl(\partial M \times (0,\infty) \bigr) = \overline{\xi^a} \oplus \R \overline{R^a} \oplus \R \partial_\rho,
\]
where $\overline{\xi^a}=\{ \overline{v} \mid v \in \xi^a \}$.

\begin{defn}\label{defn:contacttype}
Let $K$ be a compact set in $\hat{M}$ which contains $M$.
Then, $J \in \mca{J}(\hat{M})$ is \textit{of contact type} on $\hat{M}\setminus K$ with respect to $a$, if 
$\Phi_a^*J$ satisfies following:
\begin{enumerate}
\item $\Phi_a^*J$ preserves $\overline{\xi^a}$ on $\Phi_a^{-1}(\hat{M} \setminus K)$.
\item There exists $J^\infty$, an almost complex structure on $\xi^a$, such that
$d\pi|_{\overline{\xi^a}} \circ \Phi_a^*J|_{\overline{\xi^a}} = J^\infty \circ d\pi|_{\overline{\xi^a}}$ 
on $\Phi_a^{-1}(\hat{M} \setminus K)$. ($\pi$ denotes the natural projection to $\partial M$.)
\item There exists $c_J>0$ such that 
$\Phi_a^*J(\partial_\rho)= \frac{1}{\rho c_J}\overline{R^a}$ on $\Phi_a^{-1}(\hat{M} \setminus K)$.
\end{enumerate}
\end{defn}

We denote the set of $J \in \mca{J}(\hat{M})$ which are of contact type on $\hat{M}\setminus K$
with respect to $a$, by $\mca{J}_{a,K}(\hat{M})$.
Moreover, $\mca{J}_a(\hat{M}):=\bigcup_K \mca{J}_{a,K}(\hat{M})$ where $K$ runs over all compact sets in $\hat{M}$
which contain $M$.
Clearly, for two positive functions $a$ and $a'$, if $a/a'$ is a constant function then 
$\mca{J}_{a,K}(\hat{M})=\mca{J}_{a',K}(\hat{M})$.

Let $J \in \mca{J}_a(\hat{M})$, and $J^\infty$ be as in (2) in Definition \ref{defn:contacttype}.
Abbreviate the metric $\Phi_a^*\bigl(\langle\,\cdot\,,\,\cdot\,\rangle_J\bigr)$ on $\partial M \times (0,\infty)$
by $\langle\,\cdot\,,\,\cdot\,\rangle_{a,J}$.
Moreover, define a metric $\langle\,\cdot\,,\,\cdot\,\rangle_{a,J,\partial M}$ on $\partial M$ by 
\begin{itemize}
\item $\langle v, w \rangle_{a,J,\partial M} = (d\lambda^a)(v,J^\infty w)$ on $\xi^a$,
\item $\langle v, R^a \rangle_{a,J,\partial M} = 0$ for any $v \in \xi^a$,
\item $\big\lvert R^a \big\rvert_{a,J,\partial M}=c_J^{\frac{1}{2}}$.
\end{itemize}
Then, following properties are verified by simple calculation.

\begin{lem}\label{lem:metric}
\begin{enumerate}
\item On $\Phi_a^{-1}(\hat{M}\setminus K)$, $\overline{\xi^a}$, $\overline{R^a}$, $\partial_\rho$ are 
orthogonal to each other with respect to $\langle\,\cdot\,,\,\cdot\,\rangle_{a,J}$.
\item For $(z,\rho) \in \Phi_a^{-1}(\hat{M}\setminus K)$ and $v \in T(\partial M)$, 
$\big\lvert \overline{v}(z,\rho) \big\rvert_{a,J} = \rho^{\frac{1}{2}}|v|_{a,J,\partial M}$.
\item For $(z,\rho)\in \Phi_a^{-1}(\hat{M}\setminus K)$, $\big\lvert \partial_\rho(z,\rho) \big\rvert_{a,J} = (\rho c_J)^{-1/2}$.
\end{enumerate}
\end{lem}

\subsection{Floer equation}

Let $H \in \mca{H}_{\ad}(\hat{M})$, and $(J_t)_{t \in [0,1]}$ be a smooth family of elements of $\mca{J}(\hat{M})$.
For $x_-, x_+ \in \mca{C}(H)$, 
\begin{align*}
\hat{\mca{M}}_{H,(J_t)_t}(x_{-},x_{+}):=&\bigl\{u \colon \R \times [0,1] \to \hat{M} \bigm{|} \partial_s u - J_t\bigl(\partial_t u -X_H(u) \bigr)=0,\\
&\qquad\qquad u(\R \times \{0,1\}) \subset \hat{L},\, u(s)\to x_{\pm}\,(s \to \pm \infty) \bigr\}.
\end{align*}
$\hat{\mca{M}}_{H,(J_t)_t}$ admits a natural $\R$ action.
We denote the quotient by $\mca{M}_{H,(J_t)_t}$.

We also consider cases where Hamiltonians are time-dependent.
Let $H, H' \in \mca{H}_{\ad}(\hat{M})$ and $(H^s)_{s \in \R}$ be a monotone homotopy from $H$ to $H'$.
Let $(J^s_t)_{s \in \R, t \in [0,1]}$ be a smooth family of elements of $\mca{J}(\hat{M})$.
For $x_{-} \in \mca{C}(H)$ and $x_{+} \in \mca{C}(H')$, 
\begin{align*}
\hat{\mca{M}}_{(H^s,J^s_t)_{s,t}}(x_{-},x_{+}):=&\bigl\{u \colon \R \times [0,1] \to \hat{M} \bigm{|}\partial_s u - J^s_t\bigl(\partial_t u
 - X_{H^s}(u)\bigr)=0,\\
&\qquad\qquad u(\R \times \{0,1\}) \subset \hat{L},\, u(s)\to x_{\pm}\,(s \to \pm \infty) \bigr\}.
\end{align*}

For $x \in \mca{C}(H)$, we define its \textit{action} by
\[
\mca{A}_H(x) := \int_0^1 x^* \hat{\lambda} - H\bigl(x(t)\bigr) dt .
\]
The following lemma can be proved by simple calculation.

\begin{lem}\label{lem:energyformula}
For $x_{-} \in \mca{C}(H)$, $x_{+} \in \mca{C}(H')$, and $u \in \hat{\mca{M}}_{(H^s,J^s_t)_{s,t}}(x_{-},x_{+})$, 
\[
-\partial_s\bigl(\mca{A}_{H^s}\bigl(u(s)\bigr)\bigr)=
\int_0^1 \big\lvert \partial_s u(s,t) \big\rvert^2_{J^s_t} + \partial_sH^s\bigl(u(s,t)\bigr) dt.
\]
In particular, if
$\hat{\mca{M}}_{(H^s,J^s_t)_{s,t}}(x_{-},x_{+}) \ne \emptyset$, then
$\mca{A}_H(x_{-}) > \mca{A}_{H'}(x_+)$.
\end{lem}

We sometimes call elements of $\hat{\mca{M}}_{H,(J_t)_t}(x_-,x_+)$ and $\hat{\mca{M}}_{(H^s,J^s_t)_{s,t}}(x_-,x_+)$ 
\textit{Floer trajectories} from $x_-$ to $x_+$.
The next theorem asserts the existence of $C^0$ estimates for Floer trajectories.
This is proved in section 3.

\begin{thm}\label{thm:c0estimate}
\begin{enumerate}
\item
Let $H \in \mca{H}_{\ad}(\hat{M})$ and $(J_t)_{0 \le t \le 1}$ be a family of elements of $\mca{J}(\hat{M})$.
Assume that there exists a compact set $K$ in $\hat{M}$ such that 
$J_t \in \mca{J}_{a_H,K}(\hat{M})$ for any $t$.
Then, there exists a compact set $B \subset \hat{M}$ such that for any $x_-, x_+ \in \mca{C}(H)$ and 
$u \in \hat{\mca{M}}_{(H,J_t)_t}(x_-,x_+)$, $u\bigl(\R \times [0,1]\bigr) \subset B$.
\item
Let $H, H' \in \mca{H}_{\ad}(\hat{M})$ and 
$(H^s)_s$ be a monotone homotopy from $H$ to $H'$.
Let $(J^s_t)_{s,t}$ be a family of elements of $\mca{J}(\hat{M})$ such that
for sufficiently large $s_0>0$, 
\[
J^s_t=\begin{cases}
      J^{-s_0}_t &(s \le -s_0),\\
      J^{s_0}_t  &(s \ge s_0).
      \end{cases}
\]
Assume that there exists a compact set $K$ in $\hat{M}$, such that
$H^s \in \mca{H}_K(\hat{M})$ and $J^s_t \in \mca{J}_{a_{H^s},K}(\hat{M})$ for any $s, t$.
Then, there exists a compact set $B \subset \hat{M}$, such that for
any $x_- \in \mca{C}(H)$, $x_+ \in \mca{C}(H')$ and 
$u \in \hat{\mca{M}}_{(H^s,J^s_t)_{s,t}}(x_-, x_+)$, 
$u\bigl(\R \times [0,1]\bigr) \subset B$.
\end{enumerate}
\end{thm}

Finally, we state transversality results.

\begin{lem}\label{lem:transversality}

\begin{enumerate}
\item Let $H \in \mca{H}_\ad(\hat{M})$, and $K$ be a compact set in $\hat{M}$ which contains $M$.
Assume that 
$H \in \mca{H}_K(\hat{M})$ and images of all elements of $\mca{C}(H)$ are contained in $\interior K$.
Then, for generic $(J_t)_{t \in [0,1]}$, where $J_t \in \mca{J}_{a_H,K}(\hat{M})$, 
$\mca{M}_{H,(J_t)_t}(x_-,x_+)$ is a $\ind x_- - \ind x_+ -1$ dimensional smooth manifold for any
$x_-, x_+ \in \mca{C}(H)$. We denote the set of such $(J_t)_t$ by $\mca{J}_{H,K}(\hat{M})$, and
$\mca{J}_H(\hat{M}):=\bigcup_K \mca{J}_{H,K}(\hat{M})$, where $K$ runs over all compact sets in $\hat{M}$
with conditions as above.

\item Let $H, H' \in \mca{H}_\ad(\hat{M})$, $(H^s)_s$ be a monotone homotopy from $H$ to $H'$, and 
$K$ be a compact set in $\hat{M}$ which contains $M$. Assume that
$H^s \in \mca{H}_K(\hat{M})$ for any $s$, and images of all elements of $\mca{C}(H), \mca{C}(H')$ are contained in 
$\interior K$.
Then, for generic $(J^s_t)_{s \in \R, t \in [0,1]}$, where $J^s_t \in \mca{J}_{a_{H^s},K}(\hat{M})$, 
$\hat{\mca{M}}_{(H^s,J^s_t)_{s,t}}(x_{-}, x_{+})$ is a
$\ind x_- - \ind x_+$ dimensional smooth manifold for any $x_- \in \mca{C}(H), x_+ \in \mca{C}(H')$.
We denote the set of such $(J^s_t)_{s,t}$ by $\mca{J}_{(H^s)_s,K}(\hat{M})$, and
$\mca{J}_{(H^s)_s}(\hat{M}):=\bigcup_K \mca{J}_{(H^s)_s,K}(\hat{M})$, where $K$ runs over all compact sets 
in $\hat{M}$ with conditions as above.
\end{enumerate}
\end{lem}
\begin{proof}
First we prove (1).
Let $(J_t)_t$ be a family of elements of $\mca{J}_{a_H,K}(\hat{M})$.
Then, for any $x_-, x_+ \in \mca{C}(H)$ and $u \in \hat{\mca{M}}_{H,(J_t)_t}(x_-,x_+)$, 
$u^{-1}(\interior K)$ is a non-empty open set in $\R \times [0,1]$, since both
$x_-\bigl([0,1]\bigr)$ and $x_+\bigl([0,1]\bigr)$ are contained in $\interior K$.
By standard arguments (see \cite{FloerHoferSalamon}), one can perturb $(J_t)_t$ to achieve transversality conditions 
without violating the condition $J_t \in \mca{J}_{a_H,K}(\hat{M})$.
This proves (1). (2) is proved by similar arguments.
\end{proof}

\subsection{Wrapped Floer homology}
In this subsection, we define wrapped Floer homology for Liouville quadruples.
Once $C^0$ estimate for Floer trajectories is established (Theorem \ref{thm:c0estimate}),
other arguments are parallel to Lagrangian Floer theory for compact symplectic manifolds
(\cite{Floer}). 

Let $H \in \mca{H}_{\ad}(\hat{M})$, and $k$ be an integer.
Let 
\[
\mca{C}_k(H):=\bigl\{ x \in \mca{C}(H) \bigm{|} \ind x=k \bigr\},
\]
and $\WFC_k(H)$ be the free $\Z_2$ module generated over $\mca{C}_k(H)$.

Let $(J_t)_t \in \mca{J}_H(\hat{M})$. For each integer $k$, define $\partial_k^{H,(J_t)_t} \colon \WFC_k(H) \to \WFC_{k-1}(H)$ by
\[
\partial_k^{H,(J_t)_t}[x] := \sum_{y \in \mca{C}_{k-1}(H)} \sharp \mca{M}_{H,(J_t)_t} (x,y) \cdot [y].
\]

Then, 
$\bigl(\WFC_*(H),\partial_*^{H,(J_t)_t} \bigr)$ is a chain complex,
and the resulting homology group does not depend on choice of $(J_t)_t$.
We denote this homology group by $\WFH_*(H;M,\omega,X,L)$. We often abbreviate it as
$\WFH_*(H)$.

Let $H, H' \in \mca{H}_{\ad}(\hat{M})$, and $(H^s)_s$ be a monotone homotopy from $H$ to $H'$, 
and $(J^s_t)_{s,t} \in \mca{J}_{(H^s)_s}(\hat{M})$.
For each integer $k$, define $\varphi_k^{(H^s,J^s_t)_{s,t}} \colon \WFC_k(H) \to \WFC_k(H')$ by
\[
\varphi_k^{(H^s,J^s_t)_{s,t}} [x]:= \sum_{y \in \mca{C}_k(H')} \sharp \hat{\mca{M}}_{(H^s,J^s_t)_{s,t}} (x,y) \cdot [y].
\]
$\Bigl(\varphi_k^{(H^s,J^s_t)_{s,t}}\Bigr)_k$ is a chain map, hence we can define a morphism 
$\WFH_*(H) \to \WFH_*(H')$.

Let $H, H' \in \mca{H}_{\ad}(\hat{M})$.
If $a_H(z)<a_{H'}(z)$ for any $z \in \partial M$, 
then there exists a monotone homotopy $(H^s)_s$ from $H$ to $H'$, 
and morphism $\WFH_*(H) \to \WFH_*(H')$ obtained as above does not depend on choices of 
$(H^s,J^s_t)_{s,t}$.
We call this morphism \textit{monotone morphism}.

Finally, we define the wrapped Floer homology of $(M,\omega,X,L)$ by taking direct limit 
\[
\WFH_*(M,\omega,X,L):=
\varinjlim_{a_H \to \infty} \WFH_*(H).
\]

One of the important properties of wrapped Floer homology is its invariance under 
deformations. The next proposition is proved in section 3.5.

\begin{prop}\label{prop:isotopyinvarianceofWFH}
Let $(M,\omega^s,X^s,L)_{0 \le s \le 1}$ be a smooth family of Liouville quadruple.
Then there exists a canonical isomorphism $\WFH_*(M,\omega^0,X^0,L) \to \WFH_*(M,\omega^1,X^1,L)$.
\end{prop}

If $(M,\omega,X,L)$ and $(M,\omega,X',L)$ are Liouville quadruples, then 
$(M,\omega,sX+(1-s)X',L)_{0 \le s \le 1}$ is a smooth family of Liouville quadruples.
Hence, by Proposition \ref{prop:isotopyinvarianceofWFH},
$\WFH_*(M,\omega,X,L)$ does not depend on $X$. We often denote it by $\WFH_*(M,\omega,L)$.

Next corollary is easily obtained from Proposition \ref{prop:isotopyinvarianceofWFH}.

\begin{cor}\label{cor:isotopyinvarianceofWFH}
Let $(M, \omega, X, L)$ be a Liouville quadruple, and $M'$ be a compact submanifold of $\interior M$,
such that $(M',\omega|_{M'}, X|_{M'}, L \cap M')$ is also a Liouville quadruple.
Assume that there exists $H \in C^\infty(M)$ such that $dH(X)>0$ on $M \setminus \interior M'$.
Then $\WFH_*(M,\omega,L) \cong \WFH_*(M',\omega|_{M'},L \cap M')$.
\end{cor}
\begin{proof}
For any $x \in M \setminus M'$, an integral curve of $X$ through $x$ starts from $\partial M'$ and ends at $\partial M$.
This is because $\inf_{M \setminus \interior M'}dH(X)>0$.
Thus there exists a family $(M_t)_{0 \le t \le 1}$ of submanifolds of $M$ such that
$(M_t,\omega|_{M_t},X|_{M_t},L \cap M_t)_{0 \le t \le 1}$ is a smooth family of Liouville quadruples
and $M_0=M'$, $M_1=M$.
Now claim follows from Proposition \ref{prop:isotopyinvarianceofWFH}.
\end{proof}

We show an example of calculation of wrapped Floer homology.
Consider ${\R}^{2n}$ with coordinate $(q_1,\ldots,q_n,p_1,\ldots,p_n)$, and the standard symplectic form
$\omega_{\st}=\sum_{1 \le i \le n} dp_i \wedge dq_i$.
Let $D^{2n}:=\bigl\{(q,p)\bigm{|} |q|^2+|p|^2 \le 1 \bigr\},\,X:=\frac{1}{2}\sum_{1 \le i \le n} q_i \partial_{q_i} + p_i \partial_{p_i}$.
Then, 
$\bigl(D^{2n}, \omega_\st, X, D^{2n}\cap\{p=0\}\bigr)$ is a Liouville quadruple.

\begin{prop}\label{prop:WFHofdisk}
$\WFH_*\bigl(D^{2n},\omega_\st,D^{2n}\cap\{p=0\}\bigr)=0$.
\end{prop}
\begin{proof}
Let $\lambda:=i_X\omega_\st$.
Take $(a_n)_n$, an increasing sequence of positive numbers such that $\lim_{n \to \infty} a_n=\infty$ and
$a_n \notin \mca{A}(\partial D^{2n}, \lambda, \partial D^{2n}\cap\{p=0\}\bigr)$ for each $n$.

We identify $\hat{D^{2n}}$ with ${\R}^{2n}$ using a flow generated by $X$, 
and define $H_n \in \mca{H}_{\ad}(\hat{D^{2n}})$ by 
$H_n(p,q)=a_n\bigl(|p|^2+|q|^2)$.
Since $\lim_{n \to \infty}a_n=\infty$, 
\[
\WFH_*\bigl(D^{2n},\omega_\st,D^{2n}\cap\{p=0\}\bigr)=\lim_{n \to \infty} \WFH_*(H_n).
\]
The only element of $\mca{C}(H_n)$ is the constant map to $(0,\ldots,0)$, and its index goes to $\infty$
as $n \to \infty$.
Therefore, for any $k$, $\WFH_k(H_n)=0$ for sufficiently large $n$.
This completes the proof.
\end{proof}

We conclude this section with a remark on relation between wrapped Floer homology and 
Reeb chords. The following theorem can be proved by reduction to the finite dimensional Morse theory.

\begin{thm}\label{thm:WFH-}
Let $(M, \omega, X, L)$ be a Liouville quadruple. 
If $\mca{C}(\partial M, \lambda, \partial L)=\emptyset$, then
$\WFH_*(M,\omega,X,L) \cong H_*(L,\partial L)$.
\end{thm}

As a corollary, we get:

\begin{cor}\label{cor:WFH-}
Let $(M, \omega, X, L)$ be a Liouville quadruple.
If $\WFH_*(M, \omega, X, L)=0$, then
$\mca{C}(\partial M, \lambda, \partial L) \ne \emptyset$.
\end{cor}

\begin{rem}\label{rem:WFH-}
The Reeb vector field $R$ on $(\partial M, \lambda)$ depends on $\lambda$, but the characteristic foliation
$\R R$ on $\partial M$ depends only on $\omega$. 
Since the characteristic foliation determines Reeb chords up to reparametrizations, the following assertion makes sense:
if $\WFH_*(M,\omega,L)=0$, then $\mca{C}(\partial M, \partial L) \ne \emptyset$.
\end{rem}

\section{A $C^0$ estimate}

The goal of this section is to prove Theorem \ref{thm:c0estimate} and Proposition \ref{prop:isotopyinvarianceofWFH}.
Theorem \ref{thm:c0estimate} is proved in sections 3.1-3.4.
We only prove (2), since proof of (1) is much simpler than that of (2).
In section 3.1, we reduce Theorem \ref{thm:c0estimate} to three lemmas.
These lemmas are proved in sections 3.2 -- 3.4.
In section 3.5, we prove Proposition \ref{prop:isotopyinvarianceofWFH}.
The proof of Proposition \ref{prop:isotopyinvarianceofWFH} is similar to the proof of invariance of symplectic homology under deformations
(see, for instance, \cite{Seidel}).
The crucial step in the proof of Proposition \ref{prop:isotopyinvarianceofWFH} is a $C^0$ estimate for Floer trajectories
(Lemma \ref{lem:deformation}), and its proof is very similar to the proof of Theorem \ref{thm:c0estimate}.
Hence in section 3.5, we only mention few points which make difference.

\subsection{Reduction of the proof to three lemmas}

First, we introduce some abbreviations which we will use in the following of this section.
We abbreviate $a_{H^s}$ by $a^s$, 
and $\Phi_{a^s}$, $\lambda^{a^s}$, $\xi^{a^s}$, $R^{a^s}$
by $\Phi_s$, $\lambda^s$, $\xi^s$, $R^s$.
Moreover, we abbreviate $\langle\,\cdot\,,\,\cdot\,\rangle_{a^s,J^s_t}$ by $\langle\,\cdot\,,\,\cdot\,\rangle_{s,t}$, 
$\langle\,\cdot\,,\,\cdot\,\rangle_{a^s,J^s_t,\partial M}$ by $\langle\,\cdot\,,\,\cdot\,\rangle_{s,t,\partial M}$,
and $c_{J^s_t}$ by $c_{s,t}$
(see section 2.4).
Finally, we abbreviate an almost complex structure $(\Phi_s)^*(J^s_t)$ on $\partial M \times (0,\infty)$
by $\overline{J}^s_t$.

Take $\rho_0>0$ so large that $\Phi_s\bigl(\partial M \times [\rho_0,\infty)\bigr) \subset \hat{M} \setminus K$
for any $s$.
Take smooth function $\varphi:(0,\infty) \to \R$ such that
\begin{align*}
\varphi''(\rho)&\ge 0, \\
\varphi'(\rho)&=1\quad(\rho \ge \rho_0+1),\\
\varphi(\rho)&=0\quad(\rho \le \rho_0).
\end{align*}
Note that $\varphi(\rho) \ge \rho-(\rho_0+1)$ for any $\rho$.

For each $s \in \R$, we define $\varphi^s \colon \hat{M} \to \R$ by
\[
\varphi^s(x) = \begin{cases} 
               \varphi(\rho) &\bigl(x=\Phi_s(z,\rho) \bigr) \\
                       0     &(\text{otherwise})
            \end{cases}.
\]
By definition of $\rho_0$ and $\varphi$, it is easy to verify that each $\varphi^s$ is a smooth function on $\hat{M}$,
and $\varphi^s|_K \equiv 0$.

For $x_- \in \mca{C}(H)$, $x_+ \in \mca{C}(H')$ and 
$u \in \mca{M}_{(H^s,J^s_t)}(x_-, x_+)$, 
we define $\alpha^u: \R \times [0,1] \to \R$ by $\alpha^u(s,t)=\varphi^s\bigl(u(s,t)\bigr)$.

\begin{lem}\label{lem:boundary}
$\partial_t\alpha^u=0$ on $\R \times \{0,1\}$.
\end{lem}
\begin{proof}
If $u(s,t) \in K$, then $\alpha^u \equiv 0$ on some neighborhood of $(s,t)$,
hence $\partial_t\alpha^u(s,t)=0$.
Therefore it is enough to consider the case $u(s,t) \notin K$.
Let $D:=\bigl\{(s,t) \in \R \times [0,1] \bigm{|} u(s,t) \notin K \bigr\}$.
This is an open set in $\R \times [0,1]$.
Define $v: D \to \partial M \times (0,\infty)$ by
\[
v(s,t):=(\Phi_s)^{-1}\bigl(u(s,t)\bigr)
\]
and $z: D \to \partial M$, $\rho: D \to (0,\infty)$ by
\[
\bigl(z(s,t),\rho(s,t)\bigr):=v(s,t).
\]
Since $u$ satisfies $\partial_s u-J^s_t\bigl(\partial_t u-X_{H^s}(u)\bigr)=0$, by simple calculation we obtain:
\begin{equation}\label{eq:floereq_v}
\partial_s v - \overline{J}^s_t \partial_t v - \rho\cdot\bigl(c_{s,t}+ \partial_s a^s(z)\cdot a^s(z)^{-1}\bigr)\partial_\rho=0.
\end{equation}
Since $\alpha^u(s,t)=\varphi\bigl(\rho(s,t)\bigr)$, it is enough to show $d\rho(\partial_t v)=0$.
By (1) in Lemma \ref{lem:metric}, it is equivalent to $\langle \partial_t v, \partial_\rho \rangle_{s,t}=0$.
By (\ref{eq:floereq_v}), it is enough to check 
\[
\big\langle \overline{J}^s_t \partial_s v, \partial_\rho \big\rangle_{s,t}=0, \qquad
\big\langle \overline{J}^s_t\partial_\rho, \partial_\rho \big\rangle_{s,t}=0.
\]
The latter is obvious. Since $u\bigl(\R \times \{0,1\}\bigr) \subset \hat{L}$, 
if $t \in \{0,1\}$ then 
\[
\partial_s v(s,t) \in T(\partial L) \oplus \R \partial_\rho \subset \xi^s \oplus \R \partial_\rho.
\]
Hence 
 $\overline{J}^s_t \partial_s v \in \xi^s \oplus \R R^s$.
Therefore
$\overline{J}^s_t\partial_s v$ is orthogonal to $\partial_\rho$.
\end{proof}

Following three lemmas play crucial role in the proof of Theorem \ref{thm:c0estimate}.
They are proved in sections 3.2 -- 3.4.

\begin{lem}\label{lem:ellipticestimate}
For any $x_- \in \mca{C}(H)$ and $x_+ \in \mca{C}(H')$, 
there exists $c_0(x_-,x_+) ,c_1(x_-,x_+) >0$ such that
$\Delta \alpha^u + c_0(x_-,x_+) \alpha^u + c_1(x_-,x_+) \ge 0$
for every $u \in \mca{M}_{(H^s,J^s_t)_{s,t}}(x_-, x_+)$.
\end{lem}

\begin{lem}\label{lem:sequence}
For any $x_- \in \mca{C}(H)$, $x_+ \in \mca{C}(H')$ and 
$\delta>0$, there exists $c(x_-, x_+, \delta)>0$ such that:
for any $u \in \mca{M}_{(H^s,J^s_t)_{s,t}}(x_-, x_+)$, there exists a sequence $(s_k)_{k \in \Z}$ with following properties:
\begin{enumerate}
\item $0<s_{k+1}-s_k<\delta$ for any $k$.
\item $\sup_{t \in [0,1]}\alpha^u(s_k,t) \le c(x_-, x_+, \delta)$ for any $k$.
\end{enumerate}
\end{lem}

\begin{lem}\label{lem:mponstrip}
Assume that $a, b, \lambda \ge 0$ and $\delta>0$ are given such that $\delta^2\lambda<\pi^2$.
Then, there exists $c(a,b,\lambda,\delta)>0$ such that, 
if a closed interval $I$ satisfies $0<|I| \le \delta$ and 
a smooth function $\alpha \colon I \times [0,1] \to \R$ satisfies 
\begin{enumerate}
\item $\partial_t \alpha=0$ on $I \times \{0,1\}$, 
\item $\Delta \alpha + \lambda \alpha + a \ge 0$, 
\item $\sup\bigl\{ \alpha(s,t) \bigm{|} s \in \partial I\bigr\} \le b$, 
\end{enumerate}
then, $\sup \alpha \le c(a,b,\lambda,\delta)$.
\end{lem}

We give a proof of Theorem \ref{thm:c0estimate} (2) assuming those results.
Since $\mca{C}(H)$ and $\mca{C}(H')$ are finite sets, it is enough to show that:
\begin{quote}
For any $x_- \in \mca{C}(H)$ and $x_+ \in \mca{C}(H')$, there exists a compact set $B(x_-,x_+) \subset \hat{M}$ such 
that any $u \in \hat{\mca{M}}_{(H^s,J^s_t)_{s,t}}(x_-,x_+)$ satisfies $u\bigl(\R \times [0,1]\bigr) \subset B(x_-,x_+)$.
\end{quote}
Take $\delta>0$ so small that $\delta^2c_0<\pi^2$.
Then, for any $u \in \hat{\mca{M}}_{(H^s,J^s_t)_{s,t}}(x_-,x_+)$, if we take $(s_k)_k$ as in Lemma \ref{lem:sequence}, 
$u|_{I \times [s_k,s_{k+1}]}$ satisfies assumptions of Lemma \ref{lem:mponstrip} for each $k$, 
with $a=c_1, b=c(x_-, x_+, \delta), \lambda=c_0$.
(it follows from Lemma \ref{lem:boundary} and Lemma \ref{lem:ellipticestimate}).
Hence $\sup \alpha_u \le c\bigl(c_1,c(x_-, x_+, \delta),c_0,\delta\bigr)$. This proves the above claim.

\subsection{Proof of Lemma \ref{lem:ellipticestimate}}
Let $x_- \in \mca{C}(H)$ and $x_+ \in \mca{C}(H')$.
Our goal is to show that there exist $c_0, c_1>0$, which are independent of 
$u \in \hat{\mca{M}}_{(H^s,J^s_t)_{s,t}}(x_-, x_+)$, such that 
\begin{equation}\label{eq:ellipticestimate}
\Delta \alpha^u + c_0 \alpha^u + c_1 \ge 0
\end{equation}
holds on $\R \times [0,1]$.
In the following of this subsection, we fix $u$ and abbreviate $\alpha^u$ by $\alpha$.

If $u(s,t)\in K$, then $\alpha \equiv 0$ on some neighborhood of $(s,t)$, and 
(\ref{eq:ellipticestimate}) holds for any $c_0,c_1>0$.
Therefore, it is enough to show (\ref{eq:ellipticestimate}) for $(s,t) \in D$
(we use notations $D, v, z, \rho$ which are introduced in the proof of Lemma \ref{lem:boundary}).

Since $\alpha|_D = \varphi \circ \rho$, we get
\begin{equation}\label{eq:laplacian}
\Delta \alpha = \varphi''(\rho)\bigl((\partial_s\rho)^2+(\partial_t\rho)^2\bigr)+\varphi'(\rho)\Delta \rho
\ge \varphi'(\rho)\Delta\rho.
\end{equation}
Assume for the moment that there exists $c_2>0$, which is independent of $u$ and 
\begin{equation}\label{eq:rho}
\text{$\Delta \rho + c_2 \rho \ge 0$ on $D$.}
\end{equation}
Then, combining (\ref{eq:laplacian}), (\ref{eq:rho}) and 
$\varphi(\rho) \ge \rho - (\rho_0+1)$, we get
\begin{align*}
&\Delta \alpha + c_2 \alpha + c_2(\rho_0 + 1) 
\ge \Delta \alpha + c_2\varphi'(\rho)(\alpha+ \rho_0 +1)
\ge \Delta \alpha + c_2\varphi'(\rho)\rho \\
&\qquad \ge \varphi'(\rho)(\Delta \rho + c_2\rho) \ge 0.
\end{align*}
i.e. (\ref{eq:ellipticestimate}) holds for $c_0=c_2$, $c_1=c_2(\rho_0+1)$ on $D$.
Hence our goal is to show the existence of $c_2>0$ such that 
(\ref{eq:rho}) holds.

Applying $d\rho$ and $\lambda^s$ to (\ref{eq:floereq_v}), we get
\begin{align}
&\partial_s\rho+ c_{s,t}(\rho \lambda^s)(\partial_t v) - \rho \cdot \bigl(c_{s,t}+ \partial_s a^s(z) \cdot a^s(z)^{-1}\bigr)=0, \\
&c_{s,t}(\rho \lambda^s)(\partial_s v) - \partial_t \rho=0.
\end{align}
By these two equations, we get 
\begin{align*}
\Delta \rho &= c_{s,t}d(\rho \lambda^s)(\partial_t v,\partial_s v)+
\partial_s\rho\cdot\bigl(c_{s,t}+\partial_s a^s(z)\cdot a^s(z)^{-1}\bigr) \\
&+\rho \cdot \Bigl( 
\partial_s\bigl(c_{s,t}+\partial_s a^s(z)\cdot a^s(z)^{-1}\bigr)
-c_{s,t}\cdot\partial_s \lambda^s(\partial_t z)
+\partial_t c_{s,t} \cdot \lambda^s(\partial_s z)
-\partial_s c_{s,t} \cdot \lambda^s(\partial_t z)\Bigr).
\end{align*}
On the other hand, by (\ref{eq:floereq_v}), 
\[
d(\rho \lambda^s)(\partial_t v,\partial_s v) =|\partial_s v|_{s,t}^2 -c_{s,t}^{-1}\cdot \partial_s \rho \cdot\bigl(c_{s,t}+ \partial_s a^s(z)\cdot a^s(z)^{-1}\bigr).
\]
Then, we get 
\[
\Delta \rho= c_{s,t}|\partial_s v|_{s,t}^2
+\rho \cdot \Bigl( 
 \partial_s\bigl(c_{s,t}+\partial_s a^s(z)\cdot a^s(z)^{-1}\bigr)
-c_{s,t}\cdot \partial_s \lambda^s (\partial_t z)
+\partial_t c_{s,t}\cdot\lambda^s(\partial_s z)
-\partial_s c_{s,t}\cdot\lambda^s(\partial_t z)\Bigr).
\]
For $V \in T\bigl(\partial M \times (0,\infty)\bigr)$, we denote its $T(\partial M)$-part by $(V)_{\partial M}$. 
On $\Phi_s^{-1}(\hat{M} \setminus K)$, 
$T(\partial M)$ and $\partial_\rho$ are orthogonal to each other with respect to 
$\langle\,\cdot\,,\,\cdot\,\rangle_{s,t}$. Hence
$|(V)_{\partial M}|_{s,t} \le |V|_{s,t}$ for any $V$.
Then, we get (recall Lemma \ref{lem:metric}):
\begin{align*}
|\partial_s z|_{s,t,\partial M} &= \rho^{-1/2}\big\lvert (\partial_s v)_{\partial M} \big\rvert_{s,t} \le \rho^{-1/2}|\partial_s v|_{s,t},\\
|\partial_t z|_{s,t,\partial M} &= \rho^{-1/2}\big\lvert (\partial_t v)_{\partial M} \big\rvert_{s,t} \le \rho^{-1/2}|\partial_t v|_{s,t}
\le \rho^{-1/2}|\partial_s v|_{s,t} + c_{s,t}^{-1/2}\bigl(c_{s,t}+\partial_s a^s(z) \cdot a^s(z)^{-1} \bigr).
\end{align*}
In the last inequality, we use (\ref{eq:floereq_v}) and $\big\lvert \partial_\rho \big\rvert_{s,t} = (\rho c_{s,t})^{-1/2}$.
On the other hand, there exist constants $c_3, c_4, c_5>0$, which are independent of $u$ and satisfy
\[
\big\lvert \partial_s\bigl( \partial_s a^s(z)\cdot a^s(z)^{-1} \bigr)\big\rvert \le c_3|\partial_s z|_{s,t,\partial M} + c_4, \qquad
\big\lvert \partial_s\lambda^s (\partial_t z) \big\rvert \le c_5 |\partial_t z|_{s,t,\partial M}.
\]
Hence there exist constants $c_6, c_7>0$, which are independent of $u$ and satisfy
\[
\Delta \rho \ge c_{s,t}|\partial_s v|_{s,t}^2 -c_6 \rho^{1/2}|\partial_s v|_{s,t} - c_7 \rho.
\]
Therefore
\[
\Delta \rho \ge c_{s,t}|\partial_s v|_{s,t}^2 - \biggl(\frac{c_{s,t}|\partial_s v|_{s,t}^2}{2} + \frac{c_{s,t}^{-1}c_6^2\rho}{2}\biggr) - c_7\rho \ge -\biggl(\frac{c_{s,t}^{-1}c_6^2}{2} + c_7\biggr) \rho.
\]
Hence (\ref{eq:rho}) holds when $c_2 \ge \frac{\sup_{s,t} c_{s,t}^{-1} \cdot c_6^2}{2} + c_7$.
This completes the proof of Lemma \ref{lem:ellipticestimate}.

\subsection{Proof of Lemma \ref{lem:sequence}}

First note that we may replace $H^s$ with $H^s+C(s)$, where $C$ is an arbitrary smooth function on $s$. This is because 
$X_{H^s} \equiv X_{H^s+C(s)}$ for any $s$.
Therefore, we may assume that $H^s$ satisfies $\partial_sH^s(x) \ge 0$ for any $s \in \R$, $x \in \hat{M}$.

Let $u \in \mca{M}_{(H^s,J^s_t)}(x_-, x_+)$. Recall Lemma \ref{lem:energyformula}:
\[
\partial_s\bigl( \mca{A}_{H^s}\bigl(u(s)\bigr)\bigr)
=-\int_0^1 \big\lvert\partial_s u(s,t) \big\rvert_{J^s_t}^2+\partial_sH^s\bigl(u(s,t)\bigr) dt \le 0.
\]
In particular, 
\[
\mca{A}_{H'}(x_+) \le \mca{A}_{H^s}\bigl(u(s)\bigr) \le \mca{A}_H(x_-) 
\]
for any $s$.
Hence, for any interval $I \subset \R$, there exists $s \in I$ such that
\[
|I| \cdot \int_0^1 \big\lvert \partial_s u(s,t) \big\rvert_{J^s_t}^2 + \partial_s H^s\bigl(u(s,t)\bigr) dt \le
\mca{A}_H(x_-) - \mca{A}_{H'}(x_+).
\]
Hence, we can conclude:

\begin{lem}\label{lem:sequence1}
For any $\delta>0$ and $u \in \hat{\mca{M}}_{(H^s,J^s_t)}(x_-,x_+)$, 
there exists a sequence $(s_k)_{k \in \Z}$ with following properties:
\begin{enumerate}
\item $0 < s_{k+1}-s_k < \delta$ for any $k$.
\item $\int_0^1 \big| \partial_s u(s,t) \big|_{J^s_t}^2 + \partial_s H^s\bigl(u(s,t)\bigr) dt
\le \frac{2\bigl(\mca{A}_H(x_-)-\mca{A}_{H'}(x_+)\bigr)}{\delta}$ for any $k$.
\end{enumerate}
\end{lem}

Note that $|\partial_s u|_{J^s_t} =|\partial_t u - X_{H^s} \circ u|_{J^s_t}$.
Therefore, to prove Lemma \ref{lem:sequence}, it is sufficient to prove the following:

\begin{lem}\label{lem:sequence2}
For any $c>0$, there exists $M(c)>0$ such that:
if $s \in \R$ and 
$x \colon [0,1] \to \hat{M}$ satisfy $x(0), x(1) \in \hat{L}$ and
\[
\int_0^1 \big\lvert \partial_t x - X_{H^s}\bigl(x(t)\bigr) \big\rvert_{J^s_t}^2 + \partial_s H^s \bigl(x(t)\bigr)  dt
\le c,
\]
then $\sup_{0 \le t \le 1} \varphi^s\bigl(x(t)\bigr) \le M(c)$.
\end{lem}
\begin{proof}
If this lemma does not hold, there exist sequences $(s_k)_k$ and $(x_k)_k$ such that
\begin{align}
&\label{eq:bound0}\int_0^1 \big\lvert \partial_t x_k - X_{H^{s_k}}\bigl(x_k(t)\bigr) \big\rvert_{J^{s_k}_t}^2 +\partial_s H^s\bigl(s_k,x_k(t)\bigr) dt
\le c, \\
&\label{eq:infty}\lim_{k \to \infty} \sup_{0 \le t \le 1} \varphi^{s_k}\bigl(x_k(t)\bigr)= \infty.
\end{align}
Recall that in the statement of Theorem \ref{thm:c0estimate}, we take $s_0>0$ such that
\[
J^s_t = \begin{cases}
        J^{-s_0}_t &(s \le -s_0), \\
        J^{s_0}_t  &(s \ge s_0).
        \end{cases}
\]
By replacing $s_0$ if necessary, we may assume that $s_0$ also satisfies
$H^s = \begin{cases}
        H &(s \le -s_0) \\
        H'&(s \ge s_0)
       \end{cases}$.
Then, we may assume that $s_k \in [-s_0, s_0]$ for all $k$.
Note that (\ref{eq:bound0}) implies
\begin{align}
&\label{eq:bound1} \int_0^1 \big\lvert \partial_t x_k - X_{H^{s_k}}\bigl(x_k(t)\bigr) \big\rvert_{J^{s_k}_t}^2  dt \le c, \\
&\label{eq:bound2} \int_0^1 \partial_s H^s\bigl(s_k,x_k(t)\bigr) dt \le c.
\end{align}

First we show that $\lim_{k \to \infty} \inf_{0 \le t \le 1}\varphi^{s_k}\bigl(x_k(t)\bigr) = \infty$.
If this does not hold, by replacing $(s_k)_k$ and $(x_k)_k$ to their subsequences, we may assume that
$\sup_k \inf_{0 \le t \le 1}\varphi^{s_k}\bigl(x_k(t)\bigr) < \infty$.
Then, for sufficiently large $k$, there exist $a_k, b_k \in [0,1]$ such that
\begin{align*}
&\sup_k \varphi^{s_k}\bigl(x_k(a_k)\bigr) < \infty,  \\
&\lim_{k \to \infty} \varphi^{s_k}\bigl(x_k(b_k)\bigr) = \infty, \\
&0 \le \theta \le 1 \implies x_k\bigl(\theta a_k + (1-\theta)b_k \bigr) \subset \hat{M} \setminus K.
\end{align*}
Without loss of generality, we may assume that $a_k \le b_k$.
Define $y_k: [a_k,b_k] \to \partial M \times (0,\infty)$, 
$z_k: [a_k,b_k] \to \partial M$, $\rho_k: [a_k,b_k] \to (0,\infty)$ by
\[
y_k(t):=(\Phi_{s_k})^{-1}(x_k(t)), \qquad \bigl(z_k(t),\rho_k(t)\bigr):=y_k(t).
\]
Then 
\begin{align*}
&\int_{a_k}^{b_k} \big\lvert \partial_t x_k - X_{H^{s_k}}\bigl(x_k(t)\bigr) \big\rvert_{J^{s_k}_t}^2 dt
=\int_{a_k}^{b_k} \big\lvert \partial_t y_k - \overline{R^{s_k}}(y_k(t)) \big\rvert_{s_k,t}^2 dt
\ge\int_{a_k}^{b_k} \big\lvert \bigl(\partial_t y_k \bigr)_{\partial_\rho} \big\rvert_{s_k,t}^2 dt  \\
&\quad\ge \inf_{s,t} c_{s,t}^{-1} \int_{a_k}^{b_k} \bigl(\rho_k(t)^{-\frac{1}{2}}\cdot \partial_t\rho_k\bigr)^2 dt
\ge \inf_{s,t} c_{s,t}^{-1} \cdot 4 \bigl( \rho_k(b_k)^{\frac{1}{2}} - \rho_k(a_k)^{\frac{1}{2}} \bigr)^2 \cdot (b_k-a_k)^{-1}.
\end{align*}
Since $\rho_k(a_k)$ is bounded and $\lim_{k \to \infty} \rho_k(b_k)=\infty$, we get 
\[
\lim_{k \to \infty} \int_{a_k}^{b_k} \big\lvert \partial_t x_k- X_{H^{s_k}}\bigl(x_k(t)\bigr) \big\rvert_{J^{s_k}_t}^2 dt
= \infty.
\]
This contradicts (\ref{eq:bound1}), 
and we have shown that $\lim_{k \to \infty}\inf_{0 \le t \le 1}\varphi^{s_k}\bigl(x_k(t)\bigr)= \infty$.
In particular, $x_k\bigl([0,1]\bigr) \subset \hat{M} \setminus K$ for sufficiently large $k$.
For such $k$, define $y_k: [0,1] \to \partial M \times (0,\infty)$, 
$z_k: [0,1] \to \partial M$, $\rho_k \colon [0,1] \to (0,\infty)$ by
\[
y_k(t):=(\Phi_{s_k})^{-1}\bigl(x_k(t)\bigr),\qquad \bigl(z_k(t),\rho_k(t)\bigr):=y_k(t).
\]
Then, by (\ref{eq:bound1}) and (\ref{eq:bound2}), $y_k$ satisfies 
\begin{align}
&\label{eq:bound1'}\int_0^1 \big\lvert \partial_t y_k- \overline{R^{s_k}} \bigl(y_k(t)\bigr) \big\rvert_{s_k,t}^2 dt \le c, \\
&\label{eq:bound2'}\int_0^1 \partial_s a^s \bigl(s_k,z_k(t)\bigr) \cdot a^{s_k}\bigl(z_k(t)\bigr)^{-1} \cdot \rho_k(t) dt +\partial_s b(s_k) \le c.
\end{align}
Since $\lim_{k \to \infty}\inf_{0 \le t \le 1}\varphi^{s_k}\bigl(x_k(t)\bigr)= \infty$, we get 
$\lim_{k \to \infty} \inf_{0 \le t \le 1} \rho_k(t) = \infty$.

By replacing $(x_k)_k$ and $(s_k)_k$ to their subsequences, we may assume that $(s_k)_k$ converges to some $s_\infty \in [-s_0,s_0]$.
Since
\[
\Big\lvert \partial_t y_k - \overline{R^{s_k}}\bigl(y_k(t)\bigr) \Big\rvert_{s_k,t}^2
\ge \Big\lvert \bigl(\partial_t y_k- \overline{R^{s_k}}\bigl(y_k(t)\bigr)\bigr)_{\partial M} \Big\rvert_{s_k,t}^2
=\Big\lvert \partial_t z_k - R^{s_k}\bigl(z_k(t)\bigr) \Big\rvert_{s_k,t,\partial M}^2 \cdot \rho_k(t),
\]
we get from (\ref{eq:bound1'}) and $\lim_{k \to \infty} \inf_{0 \le t \le 1} \rho_k(t) = \infty$ that 
\[
\lim_{k \to \infty} \int_0^1 \Big\lvert \partial_t z_k - R^{s_k}\bigl(z_k(t)\bigr) \Big\rvert_{s_k,t,\partial M}^2  dt
=0.
\]
Then, by taking limit of certain subsequence of $(z_k)_k$, we get $z_\infty: [0,1] \to \partial M$ such that
\[
z_\infty(0), z_\infty(1) \in \partial L, \qquad \partial_t z_\infty(t)= R^{s_\infty}\bigl(z_\infty(t)\bigr).
\]
Therefore $1 \in \mca{A}(\partial M, \lambda^{s_\infty}, \partial L)$, 
hence $s_\infty \in (-s_0, s_0)$.
By the definition of monotone homotopy, 
$\inf_{z \in \partial M} \partial_s a^s(s_\infty,z)>0$.
Hence, there exists $\varepsilon>0$ such that
$\inf_{z \in \partial M} \partial_s a^s(s_k,z) \ge \varepsilon$ for sufficiently large $k$.
Let $A:= \sup_{(s,z) \in \R \times \partial M} a^s(z)$. Then, 
\[
\int_0^1 \partial_s a^s\bigl(s_k,z_k(t)\bigr)\cdot a^{s_k}\bigl(z_k(t)\bigr)^{-1} \cdot \rho_k(t) dt 
\ge \varepsilon A^{-1} \int_0^1 \rho_k(t) dt
\]
for sufficiently large $k$.
Since $\lim_{k \to \infty} \inf_{0 \le t \le 1} \rho_k(t) = \infty$, the
right hand side of the above inequality goes to $\infty$ as $k \to \infty$.
Hence the left hand side of the above inequality also goes to $\infty$ as $k \to \infty$.
This contradicts (\ref{eq:bound2'}). 
This completes the proof of Lemma \ref{lem:sequence2}.
\end{proof}

\subsection{Proof of Lemma \ref{lem:mponstrip}}

We use the following result, which is exactly the same as Proposition 8 in \cite{FloerHofer}.

\begin{lem}\label{lem:maximumprinciple}
Assume that $a, b, \lambda \ge 0$ and $\delta>0$ are given such that $\delta^2\lambda<\pi^2$.
Then, there exists $C(a, b, \lambda, \delta)>0$ such that, 
if a closed interval $I$ satisfies $0<|I| \le \delta$ and 
a smooth function $\alpha \colon I \times \R/\Z \to \R$ satisfies 
\begin{align*}
&\Delta \alpha + \lambda \alpha + a \ge 0, \\
&\sup\bigl\{ \alpha(s,t) \bigm{|} s \in \partial I\bigr\} \le b,
\end{align*}
then, $\sup \alpha \le C(a, b, \lambda, \delta)$.
\end{lem}

\begin{rem}
For any $\tau>0$, Lemma \ref{lem:maximumprinciple} holds if we replace $\R/\Z$ with $\R/\tau \Z$
in the statement.
\end{rem}

\begin{proof}[\textbf{proof of Lemma \ref{lem:mponstrip}}:]
For any $\varepsilon>0$, there exists $\delta>0$ and $\beta \colon I \times [0,1] \to \R$ such that:
\begin{align*}
&\sup_{I \times [0,1]} |\alpha - \beta| , \sup_{I \times [0,1]} \big\lvert \Delta(\alpha- \beta) \big\rvert \le \varepsilon, \\
&1-\delta \le t \le 1 \implies \beta(s,t) = \alpha(s,1) + \partial_t^2\alpha(s,1)\cdot\frac{(t-1)^2}{2},\\
&0 \le t \le \delta \implies \beta(s,t)=\alpha(s,0) + \partial_t^2\alpha(s,0) \cdot \frac{t^2}{2}.
\end{align*}
Define $\overline{\beta} \colon I \times \R/2\Z \to \R$ by
\[
\overline{\beta}(s,t) = \begin{cases}
                                   \beta(s,t) &( 0 \le t \le 1),\\
                                   \beta(s,2-t) &(1 \le t \le 2).
                                   \end{cases}
\]
Then, $\overline{\beta} \in C^\infty(I \times \R/2\Z)$.
Moreover, $\overline{\beta}$ satisfies 
\[
\Delta \overline{\beta} + \lambda \overline{\beta} + \bigl(a+(1+\lambda)\varepsilon \bigr) \ge 0, \qquad
\sup\big\{\overline{\beta}(s,t)\bigm{|} s \in \partial I \bigr\} \le b+ \varepsilon.
\]
Then, if we take $C=C(a+(1+\lambda)\varepsilon, b+\varepsilon, \lambda, \delta)$ as in Lemma \ref{lem:maximumprinciple},
$\sup \beta = \sup \overline{\beta} \le C$. Hence
$\sup \alpha \le C+\varepsilon$.
\end{proof}

\subsection{Proof of Proposition \ref{prop:isotopyinvarianceofWFH}}

First, we may assume that 
$(\omega^s,X^s)=(\omega^0,X^0)$ if $s$ is sufficiently close to $0$, and
$(\omega^s,X^s)=(\omega^1,X^1)$ if $s$ is sufficiently close to $1$.
Then, extend $(\omega^s,X^s)_{0 \le s \le 1}$ to $(\omega^s,X^s)_{s \in \R}$ by
\[
(\omega^s,X^s) = \begin{cases}
                 (\omega^0, X^0) &(s \le 0) \\
                 (\omega^1, X^1) &(s \ge 1).
                 \end{cases}
\]

The crucial step in the proof of Proposition \ref{prop:isotopyinvarianceofWFH} is:

\begin{lem}\label{lem:deformation}
Let $H, H' \in \mca{H}_{\ad}(\hat{M})$ and $(H^s)_s$ be a monotone homotopy from $H$ to $H'$, such that
$a_{H^s}$ is a constant function on $\partial M$ for any $s$ ($a_{H^s} \equiv:a(s)$).
Let $(J^s_t)_{s,t}$ be a family of almost complex structures on $\hat{M}$ such that
\[
J^s_t=\begin{cases}
       J^0_t &(s \le 0) \\
       J^1_t &(s \ge 1)
       \end{cases}.
\]
Assume that there exists a compact set $M \subset K \subset \hat{M}$, such that 
$H^s \in \mca{H}_K(\hat{M})$, 
$J^s_t \in \mca{J}_{1,K}(\hat{M};\hat{\omega^s})$ for any $s$ and $t$ (here $1$ denotes the constant function 
on $\partial M$).
Then, there exist constants $c_0, c_1>0$, which depend only on $(\omega^s,X^s)_s$ and $(J^s_t)_{s,t}$, with following
property:
if $c_0a+c_1 \le \partial_s a$ on $[0,1]$, there exists a compact set $B \subset \hat{M}$ such 
that $u\bigl(\R \times [0,1]\bigr) \subset B$ for any $x_- \in \mca{C}(H)$, $x_+ \in \mca{C}(H')$, 
$u \in \hat{\mca{M}}_{(H^s,J^s_t)}(x_-,x_+)$.
\end{lem}

Once Lemma \ref{lem:deformation} is established, we can define a chain map
$\varphi^{(H^s,J^s_t)_{s,t}}_k: \WFC_k(H;M,\omega^0,X^0,L) \to \WFC_k(H';M,\omega^1,X^1,L)$ by
\[
\varphi^{(H^s,J^s_t)_{s,t}}_k[x] = \sum_{y \in \mca{C}_k(H')} \sharp \hat{\mca{M}}_{(H^s,J^s_t)}(x,y) \cdot [y],
\]
given a monotone homotopy $(H^s)_s$ which satisfies assumptions in Lemma \ref{lem:deformation}.
Hence we get a morphism $\WFH_k(H;M,\omega^0,X^0,L) \to \WFH_k(H';M,\omega^1,X^1,L)$.
By taking direct limit, we obtain a morphism
\[
\WFH_*(M,\omega^0,X^0,L) \to \WFH_*(M,\omega^1,X^1,L).
\]
We can also define a morphism in invert direction, and show that they are inverse to 
each other. This completes the proof of Proposition \ref{prop:isotopyinvarianceofWFH}.
Hence all we have to show is Lemma \ref{lem:deformation}.

The proof of Lemma \ref{lem:deformation} is very similar to the proof of Theorem \ref{thm:c0estimate}.
First we take $\rho_0>1$ so that $K \subset \interior M(\rho_0)$, 
take smooth function $\varphi:[1,\infty) \to \R$ such that
\begin{align*}
\varphi''(\rho)&\ge 0, \\
\varphi'(\rho)&=1\quad(\rho \ge \rho_0+1),\\
\varphi(\rho)&=0\quad(\rho \le \rho_0),
\end{align*}
and define $\alpha_u \in C^\infty\bigl(\R \times [0,1]\bigr)$ for $u \in \hat{\mca{M}}_{(H^s,J^s_t)}(x_-,x_+)$ by
\[
\alpha_u(s,t) = \begin{cases}
                 \varphi\bigl(\rho(s,t)\bigr) &\bigl(u(s,t) \in \partial M \times [\rho_0,\infty)\bigr) \\
                 0 &(\text{otherwise})
                \end{cases}.
\]
Once we establish properties which correspond to 
Lemmas  \ref{lem:boundary}, \ref{lem:ellipticestimate}, \ref{lem:sequence}
for $\alpha_u$, the proof completes.
The first two properties can be proved in completely same way.
But to establish the property which corresponds to Lemma \ref{lem:sequence}, we need somewhat different arguments.
In the following, we prove the property which corresponds to Lemma \ref{lem:sequence}.
First we spell out what we have to prove.

\begin{lem}\label{lem:sequenceagain}
Let $H$, $H'$, $(H^s)_s$ and $(J^s_t)_{s,t}$ are as in Lemma \ref{lem:deformation}.
Then, there exist constants $c_0, c_1 > 0$, which depend only on $(\omega^s,X^s)_s$ and
$(J^s_t)_{s,t}$, with following property:
\begin{quote}
Assume $c_0 a(s) + c_1 \le a'(s)$ 
for $s \in [0,1]$. Then, for any $x_- \in \mca{C}(H)$, $x_+ \in \mca{C}(H')$ and $\delta>0$, 
there exists $c(x_-,x_+,\delta)>0$ such that 
for any $u \in \hat{\mca{M}}_{(H^s,J^s_t)}(x_-,x_+)$,
there exists a sequence $(s_k)_{k \in \Z}$ with :
\begin{enumerate}
\item $0<s_{k+1}-s_k<\delta$ for any $k$.
\item $\sup_{0 \le t \le 1} \alpha_u(s_k,t) \le c(x_-,x_+,\delta)$ for any $k$.
\end{enumerate}
\end{quote}
\end{lem}
\begin{proof}
At first, by same arguments as what we have done in the beginning of section 3.3, 
we may assume that $\partial_sH^s(x) \ge 0$ for any $x \in \hat{M}$ and $s \in \R$. 

Let $\hat{\lambda^s}:=i_{\hat{X^s}}\hat{\omega^s}$.
By simple calculation, we get
\begin{equation}
-\partial_s\bigl(\mca{A}_{H^s}(u(s))\bigr)
=\int_0^1 \big\lvert \partial_s u(s,t) \big\rvert_{J^s_t}^2 + \partial_sH^s\bigl(u(s,t)\bigr) - 
\partial_s\hat{\lambda^s}\bigl(\partial_t u(s,t)\bigr) dt.
\end{equation}
existence of the third term in integrand requires more arguments than proof of Lemma \ref{lem:sequence}.
In the following, we prove that: there exists $c_0, c_1>0$ such that, if 
$c_0a + c_1 \le \partial_s a$ holds on $[0,1]$, then there exists $c_2>0$ (which may depend on $(H^s)_s$)
such that
\begin{equation}\label{eq:lambdas}
\big\lvert \partial_s u(s,t) \big\rvert_{J^s_t}^2 + \partial_sH^s\bigl(u(s,t)\bigr) - \partial_s\hat{\lambda^s}\bigl(\partial_t u(s,t)\bigr) 
+c_2
\ge
\frac{1}{2}\biggl(\big\lvert \partial_s u(s,t) \big\rvert_{J^s_t}^2 + \partial_s H^s\bigl(u(s,t)\bigr)\biggr).
\end{equation}
Once this is established, Lemma \ref{lem:sequenceagain} is proved by same arguments as proof of Lemma \ref{lem:sequence}.

Since $K$ is compact, to prove (\ref{eq:lambdas}) it is enough to show that there exists $c_3>0$ such that
\begin{equation}\label{eq:c1}
u(s,t) \in \hat{M} \setminus K \implies
\big\lvert \partial_s\hat{\lambda^s}\bigl(\partial_t u(s,t)\bigr)\big\rvert
\le c_3 + \frac{1}{2}\Bigl( \big\lvert \partial_s u(s,t) \big\vert_{J^s_t}^2 + \partial_sH^s\bigl(u(s,t)\bigr) \Bigr).
\end{equation}
First notice that, since $J^s_t \in \mca{J}_{1,K}(\hat{M};\hat{\omega^s},\hat{X^s})$, 
$\langle\,\cdot\,,\,\cdot\,\rangle_{J^s_t}$ satisfies following properties (see Lemma \ref{lem:metric}):
\begin{enumerate}
\item On $\hat{M} \setminus K$, a natural decomposition $T\hat{M}=T(\partial M) \oplus \R \partial_\rho$
is an orthogonal decomposition with respect to $\langle\,\cdot\,,\,\cdot\,\rangle_{J^s_t}$.
\item There exists a metric $\langle\,\cdot\,,\,\cdot\,\rangle_{J^s_t,\partial M}$ on $\partial M$ such that
$\big\lvert \overline{v}(z,\rho) \big\rvert_{J^s_t} = \rho^{\frac{1}{2}}\big\lvert v(z) \big\rvert_{J^s_t,\partial M}$ for
any $v \in T(\partial M)$ and $(z,\rho) \in \hat{M} \setminus K$.
\item There exists $c_{s,t}$ such that $\big\lvert \partial_\rho(z,\rho)\big\rvert_{J^s_t}=(\rho c_{s,t})^{-\frac{1}{2}}$
on $\hat{M} \setminus K$.
\end{enumerate}

We return to the proof of (\ref{eq:c1}).
Since $\hat{M} \setminus K \subset \partial M \times [1,\infty)$, we can write
$u(s,t)=\bigl(z(s,t), \rho(s,t)\bigr)$.
Let $c_4:=\sup_{z,s,t} \big\lvert \partial_s\lambda^s(z)\big\rvert_{J^s_t,\partial M}$. Then,
\[
\big\lvert \partial_s\hat{\lambda^s}(\partial_t u)\big\rvert =
\rho(s,t) \big\lvert \partial_s\lambda^s(\partial_t z) \big\rvert
\le c_4 \rho(s,t) |\partial_t z|_{J^s_t,\partial M}
\le c_4 \rho(s,t)^{\frac{1}{2}} |\partial_t u|_{J^s_t}.
\]
Since $|\partial_t u|_{J^s_t} \le |\partial_s u|_{J^s_t} + \big\lvert \nabla^s_t H^s\big\rvert_{J^s_t}$ and 
$\big\lvert \nabla H^s(z,\rho) \big\rvert_{J^s_t} \le \sup_{s,t} c_{s,t} \cdot a(s) \rho^{\frac{1}{2}}$, 
There exist $c_5, c_6>0$ such that
\[
\big\lvert \partial_s\hat{\lambda^s}(\partial_t u) \big\rvert 
\le 
\frac{1}{2} |\partial_s u|_{J^s_t}^2 + \bigl(c_5 a(s) + c_6\bigr)\rho(s,t).
\]
On the other hand, $\partial_sH^s(z,\rho)=\partial_s a \cdot \rho + \partial_s b$ on $\hat{M}\setminus K$.
Hence, if $2c_5 a + 2c_6 \le \partial_s a$ and $0 \le 2c_3 + \partial_s b$ on 
$[0,1]$, (\ref{eq:c1}) holds for $s \in [0,1]$.
When $s \notin [0,1]$, the left hand side of (\ref{eq:c1}) is zero.
Hence, if $c_3 + \inf \partial_s b \ge 0$, (\ref{eq:c1}) holds for $s \notin [0,1]$.
This completes the proof of Lemma \ref{lem:sequenceagain}.
\end{proof}

\section{Classical Hamiltonian systems}
First we recall notations which are introduced in section 1.
Let $N$ be a $n$-dimensional manifold.
Then, $T^*N$ carries a natural symplectic form
$\omega_N:=\sum_{1\le i \le n} dp_i \wedge dq_i$.

Assume that $N$ carries a Riemannian metric.
Then, for $V \in C^\infty(N)$, we define $H_V \in C^\infty(T^{*}N)$ by 
$H_V(q,p)=V(q)+|p|^2/2$.
Note that $\Crit (H_V)=\Crit (V)$.

For $\xi \in \mf{X}(N)$, We define $F_\xi \in C^\infty(T^*N)$ and $\tilde{\xi} \in \mf{X}(T^*N)$ by
$F_\xi(q,p):=p(\xi_q)$ and $\tilde{\xi}:=X_{F_\xi}$.
Then, $L_{\tilde{\xi}} \omega=0$ and $\tilde{\xi}_{(q,0)}=\xi_q$.
For $a \in \R$, define $Y_a \in \mf{X}(T^*N)$ by $Y_a:=\vec{r}+a\widetilde{\nabla V}$, 
where $\vec{r}:=\sum_{1 \le i \le n} p_i \partial_{p_i}$.

\begin{lem}\label{lem:classicalhamiltoniansystem}
Let $K$ be a compact set in $T^*N$ such that $K \cap \Crit(V)=\emptyset$.
Then, $dH_V(Y_a)>0$ on $K$ for sufficiently small $a>0$.
\end{lem}
\begin{proof}
Since $\widetilde{\nabla V}(q,0)=\nabla V(q)$ and $K \cap \Crit(V)= \emptyset$, 
$dH_V(\widetilde{\nabla V})>0$ on $K \cap N$.
Let $K_0$ be a subset of $K$ defined by $dH_V(\widetilde{\nabla V}) \le 0$.
Since $K_0$ is compact and disjoint from $N$, 
$m:=\min_{K_0} |p|^2$ is positive.
Take $M>0$ so that $M> \max_{K_0} -dH_V(\widetilde{\nabla V})$, 
and take $0<a<m/M$.
Then, $dH_V(Y_a)> m-aM >0$ on $K_0$.
On the other hand, since $dH_V(\widetilde{\nabla V})>0$ and $dH_V(\vec{r}) \ge 0$ on 
$K \setminus K_0$, $dH_V(Y_a) > 0$ on $K\setminus K_0$ for any $a>0$.
\end{proof}

As in section 1, we abbreviate 
$H_V^{-1}\bigl((-\infty,h]\bigr)$ by $D_h$, and
$H_V^{-1}(h)$ by $S_h$.
If $h$ is a regular value of $H_V$ and $S_h$ is compact, 
then $(D_h, \omega_N, Y_a, D_h \cap N)$ is a Liouville quadruple for 
sufficiently small $a>0$. This is verified by applying Lemma \ref{lem:classicalhamiltoniansystem} for $K=S_h$.
The main result of this paper is the following:

\begin{thm}\label{mainthm}
Let $N$ be a Riemannian manifold, and $V \in C^\infty(N)$.
Assume that $h$ is a regular value of $V$, and $S_h$ is compact.
If $S_h \cap N \ne \emptyset$ and $D_h$ is connected, then $\WFH_*(D_h, \omega_N, D_h \cap N) = 0$.
\end{thm}

By Remark \ref{rem:WFH-}, Theorem \ref{mainthm} implies:

\begin{cor}\label{cor:mainthm}
Let $N$ and $V$ are as in Theorem \ref{mainthm}.
Then, $\mca{C}(S_h, S_h \cap N) \ne \emptyset$.
\end{cor}

Since elements of $\mca{C}(S_h, S_h \cap N)$ correspond to 
orbits of $X_{H_V}$ on $S_h$ which start from and end at $S_h \cap N$, 
Corollary \ref{cor:mainthm} implies Theorem \ref{thm:brakeorbits}.

In the remainder of this section, we reduce Theorem \ref{mainthm} to 
Lemma \ref{lem:handleattaching}.
First, we prove the following lemma:

\begin{lem}\label{lem:diffeotypeofDh}
$\WFH_*(D_h, \omega_N, D_h \cap N)$ depends only on diffeomorphism type of $D_h \cap N$.
\end{lem}
\begin{proof}
Let $K:=D_h \cap N$ and $\overline{K}:= K \cup \partial K \times [0,1]$.
Take any Riemannian metric $g$ on $\overline{K}$ and 
$W \in C^\infty(\overline{K})$ so that $0$ is a regular value of $W$ and
$K = W^{-1}((-\infty,0])$.
For such $(g,W)$, define $H_{g,W} \in C^\infty(T^{*}\overline{K})$ by
$H_{g,W}(q,p)=|p|_g^2/2+W(q)$, and let $D_{g,W}:=H_{g,W}^{-1}\bigl((-\infty,0]\bigr)$.
For $a \in \R$, let $Y_{g,W,a}:=\vec{r}+a\widetilde{\nabla_g W}$.
Then, $(D_{g,W}, \omega_{\overline{K}}, Y_{g,W,a}, K)$ is a Liouville quadruple
for sufficiently small $a>0$.

We claim that $\WFH_{*}(D_{g,W}, \omega_{\overline{K}}, K)$ does not depend on choice of $g$ and $W$.
In particular, $\WFH_{*}(D_{g,W}, \omega_{\overline{K}}, K)$ depends only on diffeomorphism type of $K$.
This is proved as follows.
Take two choices $(g_0,W_0)$ and $(g_1,W_1)$.
Let $g_t:=tg_1+(1-t)g_0$ and $W_t:=tW_1+(1-t)W_0$.
Then, when we take $a>0$ sufficiently small, 
$(D_{g_t,W_t}, \omega_{\overline{K}}, Y_{g_t,W_t,a}, K)$ is a smooth family of Liouville quadruples.
Then, the claim follows from Proposition \ref{prop:isotopyinvarianceofWFH}.

Extend the inclusion map $i \colon K \to N$ to 
an embedding $\overline{i} \colon \overline{K} \to N$.
Let $g$ be the pullback of the Riemannian metric on $N$ by $\overline{i}$, and
$W:=V \circ \overline{i} -h$.
Then, $(D_h, \omega_N, D_h \cap N)$ in Theorem \ref{mainthm}
can be identified with $(D_{g,W}, \omega_{\overline{K}}, K)$.
So, the above claim proves the lemma.
\end{proof}

We return to the proof of Theorem \ref{mainthm}.
We may assume $h=0$, and by Lemma \ref{lem:diffeotypeofDh}, we may assume that $V$ is Morse.
Then, $\Crit(V) \cap V^{-1}\bigl((-\infty,0]\bigr)$ consists of finitely many points. 
We denote it by $\{P_1,\ldots,P_l\}$.
Moreover, we may assume the following.
\begin{enumerate}
\item $V(P_1)<\cdots<V(P_l)<0$.
\item $1 \le \ind P_m \le n-1$ for $2 \le m \le l$ and $\ind P_1=0$.
\end{enumerate}
Note that we can eliminate critical points of index $n$, since $D_h \cap N$ is connected and 
its boundary is non-empty.

If $h \in \bigl(V(P_1), V(P_2)\bigr)$, $D_h \cap N$ is diffeomorphic to $D^n$.
Hence, by Lemma \ref{lem:diffeotypeofDh} and Proposition \ref{prop:WFHofdisk},
$\WFH_*(D_h, \omega_N, D_h\cap N)=0$. 

By Lemma \ref{lem:diffeotypeofDh}, if $[h,h']$ contains no critical value of $V$, 
then $\WFH_*(D_h, \omega_N, D_h \cap N) \cong \WFH_*(D_{h'}, \omega_N, D_{h'} \cap N)$.
Therefore, if we prove the following Theorem \ref{thm:handleattaching}, 
we can prove Theorem \ref{mainthm} by applying Theorem \ref{thm:handleattaching} to each critical points
$P_2,\ldots,P_m$.

\begin{thm}\label{thm:handleattaching}
Let $N$ be a $n$-dimensional Riemannian manifold, $V$ be a Morse function on $N$, and 
$P \in \Crit(V)$ with $1 \le \ind P \le n-1$.
Assume that there exists $\varepsilon>0$ such that
$\Crit(V) \cap V^{-1}\bigl([V(P)-\varepsilon, V(P)+\varepsilon]\bigr) = \{P\}$, and
$D_{V(P)+\varepsilon}$ is compact.
Then, 
\[
\WFH_*(D_{V(P)-\varepsilon}, \omega_N, D_{V(P)-\varepsilon} \cap N) \cong
\WFH_*(D_{V(P)+\varepsilon}, \omega_N, D_{V(P)+\varepsilon} \cap N).
\]
\end{thm}

In the remainder of this section, we reduce Theorem \ref{thm:handleattaching} to Lemma \ref{lem:handleattaching}.
By Morse lemma, there exists a coordinate neighborhood $U$ around $P$ and local chart $(q_1,\ldots,q_n)$ on $U$ such that
$P$ corresponds to $(0,\ldots,0)$ and  
\[
V(q)=V(P)+\bigl\{-(q_1^2+\ldots+q_k^2)+(q_{k+1}^2+\ldots+q_n^2)\bigr\}/2.
\]
Here $k=\ind P$. Denote by $\pi_N$ the natural projection $T^*N \to N$.
In the following of this paper, we often consider $\pi_N^{-1}(U)$ as a subset of ${\R}^{2n}$ 
using the coordinate $(q,p)$.

We introduce some notations which we use in the following of this paper.
First, we abbreviate $(q_1,\ldots,q_n)$ by $q$, $(p_1,\ldots,p_n)$ by $p$, and
$(p_1,\ldots,p_k)$, $(p_{k+1},\ldots,p_n)$, $(q_1,\ldots,q_k)$, $(q_{k+1},\ldots,q_n)$ by $p_-, p_+, q_-, q_+$.
Moreover, we set 
\begin{align*}
D\bigl([a,b]\bigr)&:=\bigl\{(q,p)\bigm{|}p=0,\,a \le |q|^2 \le b \bigr\}, \\
D_-\bigl([a,b]\bigr)&:=\bigl\{(q,p)\bigm{|}p=0,\,q_+=0,\,a \le |q_-|^2 \le b \bigr\}.
\end{align*}
$D\bigl((a,b]\bigr)$ etc. are defined in the same manner.

By Lemma \ref{lem:diffeotypeofDh}, we may assume that Riemannian metric on $U$ is $\sum_{1 \le i \le n} dq_i^2$.
Take $b>0$ sufficiently small so that $D\bigl([0,2b]\bigr) \subset U$ and
$\Crit(V) \cap V^{-1}\bigl([V(P)-b,V(P)+b]\bigr)=\{P\}$.

\begin{lem}\label{lem:DHVYA}
For sufficiently small $a>0$, $dH_V(Y_a)>0$ on
$H_V^{-1}\bigl([V(P)-b,V(P)+b]\bigr)\setminus\{P\}$.
\end{lem}
\begin{proof}
On $\pi_N^{-1}(U)$, we can write explicitly:
\begin{align*}
H_V(q,p)&=\frac{|p|^2-|q_-|^2+|q_+|^2}{2},\quad
dH_V(q,p)=pdp-q_{-}dq_{-}+q_{+}dq_{+},\\
Y_a(q,p)&=-aq_{-}\partial_{q_{-}}+(1+a)p_{-}\partial_{p_{-}}
         +aq_{+}\partial_{q_{+}}+(1-a)p_{+}\partial_{p_{+}}.
\end{align*}
Then, 
\[
dH_V(Y_a)=(1-a)|p_{+}|^2+(1+a)|p_{-}|^2+a|q|^2.
\]
Hence if $a \in (0,1)$, $dH_V(Y_a)>0$ on $\pi_N^{-1}(U)\setminus\{P\}$.
Therefore, to prove the claim, it is enough to show that $dH_V(Y_a)>0$ on $H_V^{-1}\bigl([V(P)-b, V(P)+b]\bigr)\setminus \pi_N^{-1}(U)$ for 
sufficiently small $a>0$. This follows from Lemma \ref{lem:classicalhamiltoniansystem}, since $H_V^{-1}\bigl([V(P)-b, V(P)+b]\bigr)\setminus \pi_N^{-1}(U)$ is compact and disjoint from 
$\Crit (V)$.
\end{proof}

For $H \in C^\infty(T^*N)$, let $S(H)$ be the set of $x:I \to T^*N$ with
$|I|>0$, $\dot{x}=X_H(x)$, $x(\partial I) \subset N$ and $x(\partial I) \cap D_-\bigl((0,b)\bigr) \ne \emptyset$.

We will show that for generic $H$, which is obtained by perturbing $H_V$, $S(H)$ is a countable set.
To put it more rigorously, we first explain the setting for perturbation.
Let $\mca{H}$ be an affine space consists of $H \in C^\infty(T^*N)$ such that
$\supp (H-H_V) \subset \bigl\{|p|^2 \le 2b \bigr\} \setminus \pi_N^{-1}\bigl(D\bigl([0,2b)\bigr)\bigr)$.
We equip $\mca{H}$ with usual $C^\infty$ topology, i.e. the topology induced by distance 
\[
d_{C^\infty}(H,H'):= \sum_{m=0}^\infty 2^{-m} \frac{|H-H'|_{C^m}}{1+|H-H'|_{C^m}}.
\]
Then, the following lemma holds.
The proof is postponed until the end of this section.

\begin{lem}\label{lem:perturb}
There exists $\mca{H}' \subset \mca{H}$, such that $\mca{H'}$ is of second category in $\mca{H}$ and 
$S(H)$ is a countable set for any $H \in \mca{H}'$.
\end{lem}

Take $a>0$ sufficiently small so that $dH_V(Y_a)>0$ on $H_V^{-1}\bigl([V(P)-b, V(P)+b]\bigr)\setminus\{P\}$.
Then, there exists $c>0$ such that if $H \in \mca{H}_U$ satisfies $d_{C^\infty}(H,H_V) \le c$, then
$dH(Y_a)>0$ on $H^{-1}\bigl([V(P)-b, V(P)+b]\bigr) \setminus \{P\}$.

By Lemma \ref{lem:perturb}, there exists $H \in \mca{H}_U$ such that
$d_{C^\infty}(H,H_V) \le c$ and $S(H)$ is countable.
Moreover, there exists $\varepsilon \in (0,b/2)$ such that
$H(x) \ne V(P)-\varepsilon$ for any $x \in S(H)$, since $S(H)$ is a countable set.

Let $D_{\pm}:=H^{-1}\bigl((-\infty,V(P)\pm\varepsilon]\bigr)$, and 
\[
\Sigma:=\partial D_{-} \cap D_-\bigl((0,b)\bigr) = \bigl\{(q,p) \bigm{|} p=q_+=0,\,|q_-|^2=2\varepsilon \bigr\}.
\]
We summarize their properties:

\begin{lem}\label{lem:Dpm}
\begin{enumerate}
\item $(D_{\pm}, \omega_N, Y_a, D_{\pm}\cap N)$ are Liouville quadruples.
\item $\WFH_*(D_{\pm}, \omega_N, D_{\pm}\cap N) \cong \WFH_*(D_{V(P)\pm\varepsilon}, \omega_N, D_{V(P)\pm\varepsilon}\cap N)$.
\item For any $x \colon I \to \partial D_{-}$ in $\mca{C}(\partial D_{-}, \partial D_{-}\cap N)$, 
$x(\partial I) \cap \Sigma = \emptyset$.
\end{enumerate}
\end{lem}
\begin{proof}
Since $dH(Y_a)>0$ on $H^{-1}\bigl([V(P)-b,V(P)+b]\bigr)\setminus\{P\}$, $Y_a$ points outwards on 
$\partial D_{\pm}$. This proves (1). 
To prove (2), for $0 \le t \le 1$ define
\[
H^t:=(1-t)H+tH_V, \qquad
D^t_\pm:=(H^t)^{-1}\bigl((-\infty,V(P)\pm\varepsilon]\bigr).
\]
Since $H^t \in \mca{H}_U$ and $d_{C^\infty}(H,H_t) \le c$, 
$(D^t_{\pm}, \omega_N, Y_a, D^t_{\pm}\cap N)_{0 \le t \le 1}$ is a smooth family of Liouville quadruples.
Hence (2) follows from Proposition \ref{prop:isotopyinvarianceofWFH}.
Finally we prove (3). If there exists $x: I \to \partial D_-$ in $\mca{C}(\partial D_-, \partial D_- \cap N)$ such 
that $x(\partial I) \cap \Sigma \ne \emptyset$, by reparametrizing $x$ we get an element of $S(H)$.
This contradicts the choice of $\varepsilon$.
\end{proof}

By (1) and (2) in Lemma \ref{lem:Dpm}, to prove Theorem \ref{thm:handleattaching} it is enough to show 
\begin{equation}\label{eq:pm}
\WFH_*(D_+, \omega_N, D_+ \cap N) \cong \WFH_*(D_-, \omega_N, D_- \cap N).
\end{equation}

Take $\mu \in C^\infty(\R)$ such that
\begin{enumerate}
\item $\mu'(t) \ge 0$.
\item $\mu(t)=\begin{cases}
              0 &(t \le 0) \\
              t-1/2&(t \ge 1)
              \end{cases}$.
\end{enumerate}
For $\delta>0$,  define $\mu_\delta \in C^\infty(\R)$ by
$\mu_\delta(t)=\frac{\delta}{2} + \delta \cdot \mu\biggl(\frac{t-2\varepsilon}{\delta}\biggr)$, and let
\[
D_{\delta}:= D_- \cup \bigl\{(q,p) \bigm{|} |q_{-}|^2 - 2\varepsilon \le |q_{+}|^2 + |p|^2  \le \mu_\delta(|q_{-}|^2) \bigr\}.
\]
Then, $D_- \subset D_\delta \subset D_+$ for sufficiently small $\delta>0$.

\begin{lem}\label{lem:Ddelta}
For sufficiently small $a>0$, 
$(D_\delta, \omega_N, Y_a, D_\delta \cap N)$ is a Liouville quadruple.
Moreover, $\WFH_*(D_\delta, \omega_N, D_\delta \cap N) \cong \WFH_*(D_+, \omega_N, D_+ \cap N)$.
\end{lem}
\begin{proof}
To prove the first assertion, it is enough to show that
$Y_a$ points strictly outwards on $\partial D_\delta$.
On $\pi_N^{-1}(U)$, 
\[
Y_a(q,p)=-aq_{-}\partial_{q_{-}}+(1+a)p_{-}\partial_{p_{-}}
         +aq_{+}\partial_{q_{+}}+(1-a)p_{+}\partial_{p_{+}}.
\]
If $a \in (0,1)$, then $-a<0$ and $1+a, a, 1-a>0$. Therefore $Y_a$ points strictly outwards on $\partial D_\delta \cap 
\pi_N^{-1}(U)$, since $\mu'_\delta(t) \ge 0$.
On the other hand, since $\partial D_\delta \setminus \pi_N^{-1}(U)= \partial D_- \setminus \pi_N^{-1}(U)$, 
$Y_a$ points outwards on $\partial D_\delta \setminus \pi_N^{-1}(U)$ for sufficiently small $a>0$.

The latter assertion follows from Corollary \ref{cor:isotopyinvarianceofWFH}, since
$dH(Y_a)>0$ on $D_+ \setminus D_\delta$ for sufficiently small $a>0$.
\end{proof}

By Lemma \ref{lem:Ddelta}, (\ref{eq:pm}) is reduced to:

\begin{lem}\label{lem:handleattaching}
$\WFH_*(D_{-}, \omega_N, D_{-} \cap N) \cong \WFH_*(D_\delta, \omega_N, D_\delta \cap N)$.
\end{lem}

Lemma \ref{lem:handleattaching} is proved in the next section.
In the remainder of this section, we prove Lemma \ref{lem:perturb}.

\begin{proof}

Define $S^-(H)$ and $S^+(H)$ by
\begin{align*}
S^-(H)&=\bigl\{x: [0,l] \to T^*N \bigm{|}l>0, \dot{x}=X_H(x),\,x(0) \in D_-\bigl((0,b)\bigr),\,x(l) \in N \bigr\},  \\
S^+(H)&=\bigl\{x: [0,l] \to T^*N \bigm{|}l>0, \dot{x}=X_H(x),\,x(0) \in N,\,x(l) \in D_-\bigl((0,b)\bigr) \bigr\}.
\end{align*}
In the following, we prove that there exists $\mca{H}^- \subset \mca{H}$ which is of second category in $\mca{H}$
and for any $H \in \mca{H}^-$, $S^-(H)$ is countable.
By parallel arguments,
we can also show that there exists $\mca{H}^+ \subset \mca{H}$ which is of second category in $\mca{H}$ and
for any $H \in \mca{H}^+$, $S^+(H)$ is countable.
Then, $\mca{H}':=\mca{H}^- \cap \mca{H}^+$ satisfies the requirements of Lemma \ref{lem:perturb}.

In the following, we prove that there exists $\mca{H}^-$ as above.
The proof consists of 9 steps.

\textbf{Step1:}
By definition of $\mca{H}$, any $H \in \mca{H}$ satisfies $H \equiv H_V$ on $\pi_N^{-1}\bigl(D\bigl([0,2b]\bigr)\bigr)$.
Hence, following (1), (2) holds for any $H \in \mca{H}$.
\begin{enumerate}
\item If $x:[0,t] \to \pi_N^{-1}\bigl(D\bigl([0,2b]\bigr)\bigr)$ satisfies $t>0$, 
$\dot{x}=X_H(x)$ and $x(0) \in D_-\bigl((0,2b]\bigr)$, then $x(t) \notin N$.
\item There exists $c>0$, which is independent of $H$ and 
such that: if $x:\R \to T^*N$ satisfies 
$\dot{x}=X_H(x)$ and $x(0) \in D_-\bigl([0,b]\bigr)$ then $x\bigl([0,c]\bigr) \subset \pi_N^{-1}\bigl(D\bigl([0,2b]\bigr)\bigr)$.
\end{enumerate}

\textbf{Step2:}
Let $\mca{B}$ be the set of $(l,x)$ where $l>0$ and $x \in L^{1,2}\bigl([0,1], T^*N\bigr)$, such that:
\begin{enumerate}
\item $x(0) \in D_-\bigl((0,b)\bigr)$, $x(1) \in N$.
\item If $\frac{1}{2} \le t \le 1-\frac{c}{l}$, then $x(t) \ne x(0)$.
\end{enumerate}
It is easily verified that $\mca{B}$ is a Banach submanifold of $(0,\infty) \times L^{1,2}\bigl([0,1], T^*N\bigr)$.
Let $\mca{E}$ be a Banach vector bundle over $\mca{B}$ defined by
$\mca{E}_{(x,l)}=L^2\bigl(x^*T(T^*N)\bigr)$.
For $H \in \mca{H}$, define $s_H \in \Gamma(\mca{E})$ by $s_H(x,l)=\dot{x}(t)-l\cdot X_H\bigl(x(t)\bigr)$.
If $s_H(x,l)=0$, then $x$ satisfies following conditions:
\begin{enumerate}
\item[(a)] $x\bigl([0,1]\bigr) \cap \bigl\{|p|^2<2b\bigr\} \setminus \pi_N^{-1}\bigl(D\bigl([0,2b]\bigr)\bigr) \ne \emptyset$.
\item[(b)] $x|_{[0,1)}$ is injective.
\end{enumerate}
By (1) in step 1 and $x(1) \in N$, 
$x\bigl([0,1]\bigr)$ is not contained in $\pi_N^{-1}\bigl(D\bigl([0,2b]\bigr)\bigr)$.
Moreover, if $x(t)=\bigl(q(t),p(t)\bigr) \in \pi_N^{-1}\bigl(D\bigl([0,2b]\bigr)\bigr)$,
\[
\big\lvert p(t) \big\rvert ^2 = 2\bigl( H\bigl(x(t)\bigr) - V\bigl(q(t)\bigr)\bigr)
=2\bigl(V\bigl(q(0)\bigr)-V\bigl(q(t)\bigr)\bigr) \le 2b-\big\lvert q(0) \big\rvert^2.
\]
(a) follows form this at once.

To prove (b), first notice that if there exists $1-\frac{c}{l}<t<1$ with $x(t)=x(0)$, then 
$x(1) \notin N$ by (1), (2) in step 1. Hence $x(t) \ne x(0)$ for $1-\frac{c}{l}<t<1$.
Hence, if $x|_{[0,1)}$ is not injective, there exists largest $0<t<1$ such that $x(t)=x(0)$, and
$t \le 1-\frac{c}{l}$. 
Moreover, if $t<\frac{1}{2}$, then $x(2t)=x(0)$ but this contradicts maximality of $t$. Hence
$\frac{1}{2} \le t \le 1-\frac{c}{l}$, but this contradicts (2) in definition of $\mca{B}$.

\textbf{Step3:} Take any almost complex structure $J$ on $T^*N$, which is compatible with $\omega_N$.
$J$ induces the associated metric and its Levi-Civita connection on $T^*N$, and also on 
$\mca{E} \to \mca{B}$.
Then,  $(\nabla s_H)_{(x,l)}: T_{(x,l)}\mca{B} \to \mca{E}_{(x,l)}$ is a Fredholm operator. 
In particular, $\Cok \nabla s_H= (\Im \nabla s_H)^{\perp} \subset \mca{E}_{(x,l)}$ is finite dimensional.
Note that the index of this operator is
\[
\dim D_-\bigl((0,b)\bigr) + \dim N + 1 - \dim T^*N = k+1-n.
\]

Let $\zeta \in \Cok \nabla s_H$, i.e. $\zeta$ is orthogonal to 
\[
\nabla_\xi (s_H)= \partial_t \xi - l(\nabla_\xi J \cdot \nabla H + J \cdot \nabla_\xi(\nabla H))
=:\partial_t \xi-l A(t)\cdot \xi(t), 
\]
for any $\xi \in L^{1,2}\bigl(x^*(T(T^*N))\bigr)$ with
$\xi(0) \in T_{x(0)}D_-\bigl((0,b)\bigr)$ and 
$\xi(1) \in T_{x(1)}N$.
Hence we obtain ($A^{*}(t)$ is the adjoint operator of $A(t)$):
\[
\bigl(\partial_t+lA^{*}(t)\bigr)\zeta(t)=0,\qquad
\zeta(0) \in \bigl(T_{x(0)}D_-\bigl((0,b)\bigr)\bigr)^\perp,\qquad
\zeta(1) \in \bigl(T_{x(1)}N\bigr)^\perp.
\]

\textbf{Step4:} We claim that if $s_H(x,l)=0$, then $\bigl(\partial_t+lA^{*}(t)\bigr)(\nabla H \circ x)=0$ .
This is verified as follows. 
If $(y,l) \in \mca{B}$ satisfies $y(0)=x(0)$ and $y(1)=x(1)$, then
$(\nabla H \circ y) \cdot s_H(y,l)=H\bigl(y(1)\bigr)-H\bigl(y(0)\bigr)=(\nabla H \circ x) \cdot s_H(x,l)$.
Hence, if $\xi \in L^{1,2}\bigl(x^*(T(T^*N))\bigr)$ satisfies $\xi(0)=0$ and $\xi(1)=0$, then
$\nabla_{\xi}(\nabla H \cdot s_H)=0$ at $(x,l)$. Since $s_H(x,l)=0$, it follows that $\nabla H \cdot \nabla_\xi(s_H)=0$.
Since this holds for any $\xi \in L^{1,2}\bigl(x^*(T(T^*N))\bigr)$ such that $\xi(0)=0$ and $\xi(1)=0$, the claim follows.

\textbf{Step5:} 
Let $m \in \Z_{\ge 2}$, and let $\mca{H}^m$ be an affine space consists of $H \in C^m(T^*N)$
such that $\supp (H-H_V) \subset \bigl\{|p|^2 \le 2b\bigr\} \setminus \pi_N^{-1}\bigl(D\bigl([0,2b)\bigr)\bigr)$.
$\mca{H}^m$ is an affine Banach space with $C^m$ norm.
Consider Banach vector bundle $\mca{H}^m\times\mca{E} \to \mca{H}^m\times\mca{B}$, and define
a section of this bundle $s \colon (H,x,l) \mapsto s_H(x,l)$.
$X_H$ is $C^{m-1}$ class vector field, hence $s$ is a $C^{m-1}$ class section.
We prove that if $s(H,x,l)=0$, then $\nabla s$ is surjective at $(H,x,l)$.
If this is not true, there exists $\zeta \in \Cok \nabla s_H(x,l)$, such that $\zeta \ne 0$ and
$\zeta \cdot (J\nabla h)\circ x=0$ for any $h \in \mca{H}^m-H_V$.
By (a) in step 2, there exists $0<t_0<t_1<1$ such that $x\bigl([t_0,t_1]\bigr) \subset \bigl\{|p|^2<2b\bigr\} \setminus
\pi^{-1}\bigl(D\bigl([0,2b]\bigr)\bigr)$.
Moreover, $x|_{[t_0,t_1]}$ is embedding by (b).
If a section $\eta$ of $x^*\bigl(T(T^*N)\bigr)|_{[t_0,t_1]}$ satisfies $\int_{t_0}^{t_1} \eta(t)\cdot\dot{x}(t) dt=0$
and $\supp \eta \subset (t_0,t_1)$, there exists $h \in \mca{H}^m-H_V$ such that
$\eta(t)=\nabla h\bigl(x(t)\bigr)$.
Hence $\zeta =a \nabla H \circ x$ on $(t_0,t_1)$ for some constant $a$.
Since $\zeta$ and $\nabla H \circ x$ both vanishes by the differential operator $\partial_t+lA^*(t)$, 
$\zeta=a \nabla H \circ x$ on $[0,1]$.
In particular, $\zeta(0)=a \nabla H(x(0))$.
Hence $a \nabla H(x(0)) \in \bigl(T_{x(0)}D_-\bigl((0,b)\bigr)\bigr)^\perp$.
On the other hand, $dH|_{T_{x(0)}D_-((0,b))} \ne 0$.
Hence we obtain $a=0$, contradicting $\zeta \ne 0$.

\textbf{Step6:} By step 4, $s^{-1}(0)$ is a $C^{m-1}$ class Banach submanifold of $\mca{H}^m \times \mca{B}$.
Consider $\pi_{\mca{H}^m}: s^{-1}(0) \to \mca{H}^m; (H,x,l) \mapsto H$.
This is a $C^{m-1}$ class Fredholm map of index $k+1-n \le 0$ (recall $k \le n-1$).
Hence by Sard-Smale theorem, the set of regular value of $\pi_{\mca{H}^m}$ (denote by $\mca{H}^m_\reg$)
 is of second category in $\mca{H}^m$.
Note that  $H \in \mca{H}^m_{\reg}$ if and only if 
$s_H: \mca{B} \to \mca{E}$ is transversal to $0$.

\textbf{Step7:}
For any $\delta>0$, let
\begin{align*}
\mca{B}(\delta):&=\bigl\{(x,l) \in \mca{B} \bigm{|} x(0) \in D_-\bigl([\delta,b-\delta]\bigr), \delta \le l \le \tfrac{1}{\delta}\bigr\},\\
\mca{H}^m_{\reg,\delta}:&= \mca{H}^m \setminus \pi_{\mca{H}^m}\bigl(\Crit(\pi_{\mca{H}^m}) \cap \mca{B}(\delta)\bigr).
\end{align*}
Obviously, $\mca{H}^m_\reg=\bigcap_{\delta>0} \mca{H}^m_{\reg,\delta}$.
We show that $\mca{H}^m_{\reg,\delta}$ is open in $\mca{H}^m$.
If $(H_n,x_n,l_n)_n$ is a sequence on $\Crit(\pi_{\mca{H}^m}) \cap \mca{B}(\delta)$ and 
$(H_n)_n$ converges to some $H_\infty$ in $\mca{H}^m$, 
then certain subsequence of $(x_n,l_n)$ converges to some $(x_\infty,l_\infty)$, hence 
$(H_\infty,x_\infty,l_\infty) \in \Crit(\pi_{\mca{H}^m}) \cap \mca{B}(\delta)$.
Therefore $\mca{H}^m \setminus \mca{H}^m_{\reg,\delta}$ is closed in $\mca{H}^m$.

\textbf{Step8:}
For any $\delta>0$, let 
$\mca{H}_{\reg,\delta}:=\mca{H}^m_{\reg,\delta} \cap \mca{H}$ (this does not depend on $m$).
We show that $\mca{H}_{\reg,\delta}$ is open dense set in $\mca{H}$.
Openness is clear since 
$\mca{H}^m_{\reg,\delta}$ is open in $\mca{H}^m$ and the inclusion map $\mca{H} \to \mca{H}^m$ is continuous.

To show that $\mca{H}_{\reg,\delta}$ is dense in $\mca{H}$, 
first notice that $\mca{H}^m_{\reg,\delta}$ is dense in $\mca{H}^m$ by step 6.
Hence for any $H \in \mca{H}$, there exists $H_m \in \mca{H}^m_{\reg,\delta}$ such that
$|H-H_m|_{C^m} \le 2^{-m}$. 
Since $\mca{H}^m_{\reg,\delta}$ is open in $\mca{H}^m$, there exists $0<c<2^{-m}$ such that
$c$-neighborhood of $H_k$ with respect to $|\,\cdot\,|_{C^m}$ is contained in $\mca{H}^m_{\reg,\delta}$.
Then, take $H'_m \in \mca{H}$ so that $|H_m-H'_m|_{C^m} <c$, then 
$H'_m \in \mca{H}_{\reg,\delta}$ and $|H-H'_m|_{C^m} <2^{1-m}$, hence
$\lim_{m \to \infty} H'_m=H$ in $\mca{H}$.
This shows that $\mca{H}_{\reg,\delta}$ is dense in $\mca{H}$.

\textbf{Step9:}
Let $\mca{H}_\reg:=\bigcap_{\delta>0} \mca{H}_{\reg,\delta}$.
$\mca{H}_\reg$ is of second category in $\mca{H}$ by step 8.
Note that $H \in \mca{H}_\reg$ if and only if $s_H: \mca{B} \to \mca{E}$ is transversal to $0$.
Since virtual dimension of $s_H^{-1}(0)$ is $1+k-n \le 0$, 
$s_H^{-1}(0)$ is a countable set for any $H \in \mca{H}_\reg$.
Therefore it is enough to show that if $s_H^{-1}(0)$ is countable, 
then $S^-(H)$ is countable.

Let $S^-_0(H):=\{x \in S^-(H) \mid \text{$x$ is injective}\}$, and
$S^-_1(H):=S^-(H) \setminus S^-_0(H)$.
$S^-_0(H)$ is countable, since there exists injection $S^-_0(H) \to s_H^{-1}(0)$ which maps 
$x: [0,l] \to T^*N$ to $[0,1] \to T^*N; t \mapsto x(tl)$.
Hence it is enough to show that $S^-_1(H)$ is countable.
Take $x \in S^-_1(H)$. Since $x$ is not constant, there exists smallest $0<t<l$ such that $x(t)=x(0)$.
Then $(y,t) \in s_H^{-1}(0)$ where $y \colon [0,1] \to T^*N; \tau \mapsto x(t\tau)$.
Moreover, there are only countably many $\theta>0$ such that $x(\theta) \in N$.
Hence we obtain map $S^-_1(H) \to s_H^{-1}(0)$, such that preimage of each element of $s_H^{-1}(0)$ is countable.
Therefore, $S^-_1(H)$ is countable. This completes the proof.
\end{proof}

\section{Handle attaching}
In this section, we prove Lemma \ref{lem:handleattaching}.
In 5.1, we prove a preliminary lemma on Floer trajectories (Lemma \ref{lem:annuluslemma}).
In 5.2, we give a proof of \ref{lem:handleattaching}.
\subsection{Lemma on Floer trajectories}
\begin{lem}\label{lem:annuluslemma}
Let $(M,\omega,X,L)$ be a Liouville quadruple, and $\lambda:=i_X\omega$.
Let $M^{\inn}$ be a compact submanifold of $M$ such that 
$(M^\inn, \omega|_{M^\inn}, X|_{M^\inn}, L \cap M^\inn)$ is a Liouville quadruple.
We denote the Reeb vector field and the contact distribution on $(\partial M^\inn, \lambda)$ 
by $R^\inn$, $\xi^\inn$.

Let $H, H' \in \mca{H}_{\ad}(\hat{M})$, and $(H^s)_s$ be a monotone homotopy from $H$ to $H'$ such that
$\partial_s H(x) \ge 0$ for any $s \in \R, x \in \hat{M}$.
Let $(J^s_t)_{s,t}$ be a family of elements of $\mca{J}(\hat{M})$.

Assume that there exists 
$a \in C^\infty(\R)$ and $0<\nu<1$ with following properties:
\begin{enumerate}
\item $H^s(z,\rho)=a(s)(\rho-\nu)$ on $M^\inn \setminus M^\inn(\nu^{\frac{1}{2}})$.
\item For any $s \in \R$ and $t \in [0,1]$, 
$J^s_t$ preserves $\overline{\xi^\inn}$ and $J^s_t(\partial_\rho)=\rho^{-1}\overline{R^\inn}$ on 
$M^\inn \setminus M^\inn(\nu^{\frac{1}{2}})$.
\end{enumerate}
Assume that $x_- \in \mca{C}(H)$, $x_+ \in \mca{C}(H')$ satisfy
$x_-\bigl([0,1]\bigr), x_+\bigl([0,1]\bigr) \subset M^\inn$.
Then, for any $u \in \hat{\mca{M}}_{(H^s,J^s_t)_{s,t}}(x_-, x_+)$, 
$u\bigl(\R \times [0,1]\bigr) \subset M^\inn$.
\end{lem}

Following proof is based on \cite{AbouzaidSeidel}, section 7. 

\begin{proof}
Take $s_0>0$ so that $H^s = \begin{cases}
                             H^{-s_0} &(s \le -s_0) \\
                             H^{s_0}  &(s \ge s_0)
                            \end{cases}$.
Since $H, H' \in \mca{H}_\ad(\hat{M})$, $a(-s_0), a(s_0) \notin \mca{A}(\partial M^\inn, \lambda^\inn,\partial L^\inn)$.
Hence $x_-\bigl([0,1]\bigr), x_+\bigl([0,1]\bigr) \subset M^\inn(\nu^{\frac{1}{2}})$.
We claim that $u\bigl(\R \times [0,1]\bigr) \subset M^\inn(\nu^{\frac{1}{2}})$ for any
$u \in \hat{\mca{M}}_{(H^s,J^s_t)_{s,t}}(x_-,x_+)$.
First notice that for any 
$\rho \in (\nu^{\frac{1}{2}},1]$, 
$D_\rho:= \R \times [0,1] \setminus u^{-1}\bigl(\interior M^\inn(\rho) \bigr)$ is a compact set.
If the claim is not true, there exists $\rho \in (\nu^{\frac{1}{2}},1]$ such that
$D_\rho \ne \emptyset$.
For generic $\rho$, $u$ and $u|_{\R \times \{0,1\}}$is transverse to $\partial M^\inn \times \{\rho\}$, hence
we may assume that $D_\rho$ is a compact surface with boundaries and corners.

Let
\[
\partial_H D_\rho: = \partial D_\rho \cap \R \times \{0,1\}, \qquad
\partial_V D_\rho: = \partial D_\rho \cap \R \times (0,1).
\]
It is easily verified that $\partial_su$ is not constantly $0$ on $D_\rho$.
This implies
\[
\int_{D_\rho} |\partial_s u|_{J^s_t}^2\,dsdt >0.
\]

Since $u$ satisfies the Floer equation $\partial_s u - J^s_t\partial_t u - \nabla^s_t H^s=0$, 
\begin{align*}
&\int_{D_\rho} |\partial_s u|_{J^s_t}^2 + \partial_s H^s\bigl(u(s,t)\bigr) ds dt
=\int_{D_\rho} \hat{\omega}(\partial_t u, \partial_s u) + dH^s(\partial_s u) + 
            \partial_s H^s\bigl(u(s,t)\bigr) dsdt \\
&=\int_{\partial D_\rho} -u^*\hat{\lambda} + H^s\bigl(u(s,t)\bigr) dt.
\end{align*}

We calculate the last term. First we calculate the integration on $\partial_H D_\rho$:
\[
\int_{\partial_H D_\rho} -u^*\hat{\lambda} + H^s\bigl(u(s,t)\bigr) dt
=\int_{\partial_H D_\rho} -u^*\hat{\lambda} =0.
\]
The first equality follows from $dt|_{\partial_H D^\rho}=0$, and the second equality follows from 
$u(\partial_H D_\rho) \subset \hat{L}$ and $\hat{\lambda}|_{\hat{L}}=0$.
On the other hand, since $u(\partial_VD_\rho) \subset \partial M^\inn \times \{\rho\}$, we get
\[
(s,t) \in \partial_V D_\rho \implies
H^s\bigl(u(s,t)\bigr) = a(s)\bigl(\rho-\nu \bigr), \quad
 \hat{\lambda}\bigl(X_{H^s}(u(s,t))\bigr) = a(s)\rho.
\]
Therefore
\[
\int_{\partial_V D_\rho} -u^*\hat{\lambda} + H^s\bigl(u(s,t)\bigr) dt
=\int_{\partial_V D_\rho} \hat{\lambda}(X_{H^s} \otimes dt - du) - \nu \int_{\partial_V D_\rho} a(s) dt.
\]
On the other hand, Floer equation is equivalent to
\[
J^s_t \circ (X_{H^s} \otimes dt -du)=(du - X_{H^s} \otimes dt) \circ j,
\]
where $j$ is a complex structure on $\R \times [0,1]$, defined by $j(\partial_s)=\partial_t$.
Therefore 
\[
\int_{\partial_V D_\rho} \hat{\lambda}(X_{H^s} \otimes dt-du)
=-\int_{\partial_V D_\rho} \hat{\lambda}\bigl(J^s_t \circ(du-X_{H^s} \otimes dt) \circ j\bigr).
\]
$\hat{\lambda}(J^s_t \circ X_{H^s})=-\hat{\lambda}(\nabla^s_t H^s)=0$ on $\partial M^\inn \times \{\rho\}$.
Moreover, if $V$ is a vector tangent to $\partial_V D_\rho$, and positive with respect to the boundary orientation,
then $jV$ points inwards, hence $d\rho(jV) \ge 0$.
Hence $\hat{\lambda}\bigl(J^s_t \circ du \circ j)(V) \ge 0$.
Therefore, 
\[
\int_{\partial_V D_\rho} \hat{\lambda}(X_{H^s} \otimes dt-du) \le 0.
\]
Finally,
\[
\int_{D_\rho} |\partial_s u|_{J^s_t}^2 + \partial_s H^s\bigl(u(s,t)\bigr)\,ds dt
\le - \nu \int_{\partial_V D_\rho} a(s) dt
= - \nu \int_{D_\rho} \partial_s a(s,t)\,dsdt.
\]
Since $\partial_s H^s \ge 0$ and
$\partial_s a \ge 0$(this follows from (1) and $\partial_s H^s \ge 0$),
 this implies
\[
\int_{D_\rho} |\partial_s u|_{J^s_t}^2\,dsdt \le 0.
\]
This is a contradiction.
\end{proof}

\subsection{Handle attaching}
In this subsection, we give a proof of Lemma \ref{lem:handleattaching}.
At first, we need the following lemma, which is easily proved by Moser's trick.

\begin{lem}\label{lem:Moser}
Let $X$ be a manifold and $Y$ be a submanifold of $X$.
Let $(\lambda_t)_{0 \le t \le 1}$ be a smooth family of contact forms on $X$ such that
$\lambda_t|_Y=0$ and $d\lambda_t=d\lambda_0$ for any $t$.

Then, for any compact set $K$ in $Y$, there exists $V$, a neighborhood of $K$ in $X$ ,
and $(\psi_t)_{0 \le t \le 1}$, a smooth family of embeddings from $V$ to $X$ with following properties:
\begin{enumerate}
\item $\psi_0$ is the inclusion map $V \to X$.
\item $\psi_t^*\lambda_0=\lambda_t$.
\item $\psi_t^{-1}(Y)= V \cap Y$.
\item $\psi_t|_{V \cap Y}$ is the inclusion map $V \cap Y \to X$.
\end{enumerate}
\end{lem}
\begin{proof}
First we show that there exists $W$, a neighborhood of $K$ in $X$, and 
$(\xi_t)_t$, a family of vector fields on $W$ such that 
$L_{\xi_t}\lambda_t+\partial_t \lambda_t=0$
and $\xi_t=0$ on $W \cap Y$.

Take $W$, a neighborhood of $K$ in $X$ so that the restriction morphism
$H^*_\dR(W) \to H^*_\dR(W \cap Y)$ is an isomorphism.
Since $d\lambda_t = d\lambda_0$ for any $t$, 
$\partial_t \lambda_t$ is a closed form.
Moreover, $\partial_t\lambda_t\big\vert_{Y} = 0$ since 
$\lambda_t|_Y = 0$ for any $t$.
Hence $(\partial_t\lambda_t)_t$ is a smooth family of 
exact one forms on $W$.
Hence there exists $(f_t)_t$, a family of $C^\infty$ functions on $W$ such that
$df_t= \partial_t \lambda_t$.
We may assume that $f_t$ vanishes on $Y$, since $\partial_t\lambda_t$ vanishes
on $Y$ and $H^0_\dR(W) \to H^0_\dR(W \cap Y)$ is an isomorphism.

Let $R_t$ be the Reeb vector field of $(X,\lambda_t)$ and
$\xi_t:=-f_tR_t$. Then, $\xi_t$ vanishes on $Y$ and
\[
L_{\xi_t}\lambda_t = i_{\xi_t}(d\lambda_t) + d(i_{\xi_t}\lambda_t) = -df_t = -\partial_t \lambda_t.
\]

Integrating $(\xi_t)_t$, we obtain $(\varphi_t)_t$, a family of embeddings from certain neighborhood of $K$ to $X$.
Then, $\varphi_t^*\lambda_t=\lambda_0$.
Finally, if we take $V$ sufficiently small, $\psi_t:=(\varphi_t)^{-1}|_V$ can be defined for all $0 \le t \le 1$ and 
satisfies the condition of the lemma.
\end{proof}

From now on, we start the proof of Lemma \ref{lem:handleattaching}, and we continue to use notations 
introduced in section 4.
For sufficiently small $\delta>0$, we define subsets of $\pi_N^{-1}(U)$, $A^-_\delta, A^+_\delta, B_\delta, C_\delta$ by
\begin{align*}
A^-_\delta&=\bigl\{(q,p) \bigm{|} |p|^2+|q_+|^2=|q_-|^2-2\varepsilon < \delta \bigr\}, \\
A^+_\delta&=\bigl\{(q,p) \bigm{|} |p|^2+|q_+|^2=\mu_\delta\bigl(|q_-|^2\bigr) < \delta \bigr\},\\
B_\delta&=\bigl\{(q,p) \bigm{|} |p|^2+|q_+|^2=|q_-|^2-2\varepsilon=\delta \bigr\}, \\
C_\delta&=\bigl\{(q,p) \bigm{|} |q_-|^2-2\varepsilon \le |p|^2+|q_+|^2 < \mu_\delta\bigl(|q_-|^2\bigr)\bigr\} \cup A^+_\delta.
\end{align*}
Recall that we have considered $\pi_N^{-1}(U)$ as a subset of ${\R}^{2n}$ using 
coordinate $(q,p)$.
Hence we consider these sets also as subsets of ${\R}^{2n}$.

We have shown in Lemma \ref{lem:Ddelta} that $(D_\delta, \omega_N, Y_a, D_\delta \cap N)$ is a Liouville quadruple
for sufficiently small $a$.
In the following of this paper, we fix such $a$ and denote it by $a_0$. 

Take arbitrary smooth function $a$ on $[0,1]$ such that $a(0)=a_0$ and $a(1)=\frac{1}{2}$.
By Lemma \ref{lem:Moser}, there exists $V$, a neighborhood of $\Sigma$ in 
$\partial D_-$, and 
$(\psi_t)_t$, a family of embeddings from $V$ to $\partial D_-$ with following properties:
\begin{enumerate}
\item $\psi_0$ is the inclusion map $V \to \partial D_-$.
\item $\psi_t^*\lambda_{a_0} = \lambda_{a(t)}$. ($\lambda_a$ denotes $i_{Y_a}\omega_N$.)
\item $\psi_t^{-1}(\partial D_- \cap N) =V \cap N$.
\item $\psi_t|_{V \cap N}$ is the inclusion map $V \cap N \to \partial D_-$.
\end{enumerate}

Since $\bigcap_{\delta>0} A^-_\delta = \Sigma$, $A^-_\delta \subset V$ for sufficiently small $\delta>0$.
If $A^-_\delta \subset V$, 
$\bigl(C_\delta,\omega_\st,Y_{a(t)},C_\delta \cap N\bigr)$ is glued to
$(D_-, \omega_N, Y_{a_0}, D_- \cap N)$ by $\psi_t|_{C_\delta \cap \partial D_-}$.
As a result, we get a Liouville quadruple.
We denote it by $\bigl(C_\delta \cup_{\psi_t} D_-, \omega_t, Z_t, L_t\bigr)$.

We make two remarks which are clear from constructions:
\begin{rem}\label{rem:handleattaching}
\begin{enumerate}
\item $\partial (C_\delta \cup_{\psi_t} D_-) = \bigl(\partial D_- \setminus \psi_t(A^-_\delta)\bigr) \cup A^+_\delta$.
\item For any $\delta, \delta'>0$, $\widehat{C_\delta \cup_{\psi_t} D_-}$ and
$\widehat{C_{\delta'} \cup_{\psi_t} D_-}$ can be identified naturally.
\end{enumerate}
\end{rem}

It is clear from construction that $\bigl(C_\delta \cup_{\psi_0} D_-, \omega_0, Z_0, L_0\bigr)$ 
is isomorphic to $(D_\delta,\omega_N,Y_{a_0},D_\delta \cap N)$ as Liouville quadruple.
Hence, by Proposition \ref{prop:isotopyinvarianceofWFH}, to prove Lemma \ref{lem:handleattaching}
it is enough to show that
\begin{equation}\label{eq:L1}
\WFH_*\bigl(C_\delta \cup_{\psi_1} D_-, \omega_1, L_1\bigr) \cong \WFH_*(D_-, \omega_N, D_- \cap N).
\end{equation}

Let $(\alpha_i)_i$ be an increasing sequence of positive numbers, such that
$\lim_{i \to \infty} \alpha_i = \infty$ and 
$\alpha_i \notin \mca{A}(\partial D_-, \lambda_{a_0}, \partial D_-\cap N)$.
Let $\nu \in (0,1)$, and take $F_i \in \mca{H}_{\ad}(\hat{D_-})$ such that:
\begin{enumerate}
\item[$F$-(1):] $F_1(x)<F_2(x)<\cdots$ for any $x \in \hat{D_-}$.
\item[$F$-(2):] $F_i(z,\rho)=\alpha_i(\rho-\nu)$ on $\partial D_- \times [\nu^{\frac{1}{2}},\infty)$.
\end{enumerate}

Since $a_{F_i}=\alpha_i \to \infty$ as $i \to \infty$, 
\begin{equation}\label{eq:WFH(D_-)}
\WFH_*(D_-,\omega_N,D_-\cap N) = \lim_{i \to \infty} \WFH_*(F_i)
\end{equation}

Hence to prove (\ref{eq:L1}), it is enough to show
\begin{equation}\label{eq:goal}
\WFH_{\le m}\bigl(C_\delta \cup_{\psi_1} D_-, \omega_1, L_1\bigr) \cong \lim_{i \to \infty} \WFH_{\le m}(F_i)
\end{equation}
for each positive integer $m$.
In the following, we fix $\delta$ and denote it by $\delta_0$.

Denote the Reeb vector field on $(\partial D_-, \lambda_{a_0})$ by $R$.

\begin{lem}\label{lem:delta}
For any $\alpha>0$, there exists $\delta(\alpha)>0$ such that
any $\delta \in \bigl(0, \delta(\alpha)\bigr)$ satisfies following:
\begin{quote}
Assume that $x:I \to \partial D_-$ satisfies $\dot{x}=R(x)$, $x(\partial I) \subset \psi_1(B_\delta) \cup (\partial D_- \cap N)$, 
$x(\partial I) \cap \psi_1(B_\delta) \ne \emptyset$ and 
$x(I)$ is not contained in $\psi_1(A^-_{\delta_0})$.
Then, $|I|>\alpha$.
\end{quote}
\end{lem}
\begin{proof}
Assume the assertion is not true.
Then, there exists $y: J \to \partial D_-$ such that
$\dot{y}=R(y)$, 
$y(\partial J) \subset \bigcap_{\delta>0}\psi_1\Bigl(\overline{A^-_\delta}\Bigr)\cup(\partial D_- \cap N)$, 
$y(\partial J) \cap \bigcap_{\delta>0}\psi_1\Bigl(\overline{A^-_\delta}\Bigr) \ne \emptyset$, and
$y(J)$ is not contained in $\psi_1(A^-_{\delta_0})$ (we use $B_\delta \subset \overline{A^-_\delta}$).
On the other hand,
\[
\bigcap_{\delta>0}\psi_1\Bigl(\overline{A^-_\delta}\Bigr)
=\psi_1\Biggl(\bigcap_{\delta>0} \overline{A^-_\delta}\Biggr)
=\psi_1(\Sigma)=\Sigma.
\]
In the last equality, we use property (4) of $\psi_t$.
Hence $y(\partial J) \subset \partial D_- \cap N$, and $y(\partial J) \cap \Sigma \ne \emptyset$.
Since $\Sigma \subset \psi_1(A^-_{\delta_0})$, $y$ is not constant and $|J|>0$. 
Hence $y \in \mca{C}(\partial D_-, \lambda_{a_0},\partial D_- \cap N)$.
But this contradicts (3) in Lemma \ref{lem:Dpm}.
\end{proof}

We can take sequences $(\delta_i)_i$ and $(G_i)_i$, 
where $\delta_i \in \R_{>0}$ and $G_i \in \mca{H}_{\ad}(\widehat{C_{\delta_0}\cup_{\psi_1}D_-})$,
such that $(\delta_i)_i$ satisfies following properties:
\begin{enumerate}
\item[$\delta$-(1):] $0<\delta_i<\min\bigr\{\delta_0, \delta(\alpha_i) \bigr\}$.
\item[$\delta$-(2):] $\delta_1 > \delta_2 > \cdots$.
\item[$\delta$-(3):] $\lim_{i \to \infty} \delta_i = 0$.
\end{enumerate}
and $(G_i)_i$ satisfies following properties:
\begin{enumerate}
\item[$G$-(1):] $G_i|_{D_-} = F_i|_{D_-}$.
\item[$G$-(2):] $a_{G_i} \to \infty$ as $i \to \infty$.
\item[$G$-(3):] There exists a sequence $i_1<i_2<\cdots$ such that $G_{i_1}(x) < G_{i_2}(x) < \cdots$ for any 
$x \in \widehat{C_{\delta_0} \cup_{\psi_1} D_-}$.
\item[$G$-(4):] By Remark \ref{rem:handleattaching}, there exists an embedding 
$\bigl(\partial D_- \setminus \psi_1(A^-_\delta)\bigr) \times [1,\infty) \to \widehat{C_{\delta_0} \cup_{\psi_1} D_-}$,
and we identify $\bigl(\partial D_- \setminus \psi_1(A^-_\delta)\bigr) \times [1,\infty)$ with its image.
Then, $G_i(z,\rho)=\alpha_i(\rho-\nu)$ on $\bigl(\partial D_- \setminus \psi_1(A^-_\delta)\bigr) \times [1,\infty)$.
\item[$G$-(5):] Let $\Psi: C_{\delta_0}\cup A^+_{\delta_0}\times [1,\infty) \to \R^{2n}$ be the embedding such that
$\Psi|_{C_{\delta_0}}$ is the inclusion map $C_{\delta_0} \to \R^{2n}$, and 
$\partial_\rho \Psi(z,\rho)=\rho^{-1}Y_{\frac{1}{2}}$ on $A^+_{\delta_0} \times [1,\infty)$.
We identify $C_{\delta_0} \cup A^+_{\delta_0} \times [1,\infty)$ with its 
image via $\Psi$ (note that $\Psi^*\omega_{\st}=\omega_1$).
Then, on $C_{\delta_0} \cup A^+_{\delta_0} \times [1,\infty)$, 
$G_i$ satisfies following properties (a)-(c), with respect to the coordinates on $\R^{2n}$:
\begin{enumerate}
\item[(a)] There exists $g_i \in C^\infty(\R_{\ge 0})$, such that $G_i(q,p)=g_i\bigl(|q_+|^2+|p_+|^2\bigr)$ if 
$|q_-|=|p_-|=0$. Moreover, $g_i'(t) \notin \frac{\pi}{2}\Z$ for any $t \in \R_{\ge 0}$.
\item[(b)] For $1 \le j \le k$, $\partial_{p_j}G_i/p_j >0$ if $p_j \ne 0$, and 
$\partial_{q_j}G_i/q_j <0$ if $q_j \ne 0$.
\item[(c)] There exist $A_i>\frac{(m+k)\pi}{2}$, $B_i>0$ and $C_i<0$ such that 
\[
G_i(q,p)=G_i(0,\ldots,0)+A_i\bigl(|p_+|^2+|q_+|^2) + B_i|p_-|^2 + C_i|q_-|^2
\]
on some neighborhood of $(0,\ldots, 0)$.
\end{enumerate}
\end{enumerate} 

\begin{rem}\label{rem:Gi}
The idea for construction of $(G_i)_i$ is as follows:
first, we define $G'_i: \widehat{C_{\delta_0} \cup_{\psi_1} D_-} \to \R$ by 
\[
G'_i(x) = \begin{cases}
           F_i(x) &( x \in D_-) \\
           \alpha_i(1-\nu) &(x \in C_{\delta_i}) \\
           \alpha_i(\rho-\nu) &\bigl(x = (z,\rho) \in \partial(C_{\delta_i}\cup_{\psi_1}D_-) \times [1,\infty)\bigr)
          \end{cases}.
\]
Then, $(G'_i)_i$ satisfies $G$-(1) to (4) though it is not smooth.
The idea is to replace $G'_i$ with $G_i$, which is smooth and satisfies $G$-(5), without violating $G$-(1) to (4).
This is achieved by elementary arguments, but we do not try to spell out details.
\end{rem}

The properties (a)-(c) in $G$-(5) look complicated, but they are necessary to show:

\begin{lem}\label{lem:G5}
If $x \in \mca{C}(G_i)$ satisfies $x\bigl([0,1]\bigr) \subset C_{\delta_0} \cup A_{\delta_0}^+ \times [1,\infty)$, 
then $x$ is the constant map to $(0,\ldots,0)$ and $\ind x>m$.
\end{lem}
\begin{proof}
First we show that $x\bigl([0,1]\bigr) \subset \{p_-=q_-=0\}$. Denote $x(t)=\bigl(q(t), p(t)\bigr)$ and 
consider $E(t)=q_-(t) \cdot p_-(t)$. 
Then $\partial_t E= \sum_{1 \le j \le k} \partial_{p_j}G_i \cdot p_j - \partial_{q_j}G_i \cdot q_j$.
By (b) in $G$-(5), for each $t \in [0,1]$, $\partial_t E(t) \ge 0$ and equality holds if and only if $p_-(t)=q_-(t)=0$.
On the other hand, $E(0)=E(1)=0$ since $x(0), x(1) \in \{p=0\}$. Hence $p_-(t)=q_-(t)=0$ for any $t \in [0,1]$.
By (a) in $G$-(5), 
$X_{G_i}(q,p)= 2g_i'\bigl(|p_+|^2+|q_+|^2\bigr) (p_+ \partial_{q_+} - q_+ \partial_{p_+})$ on $\{p_-=q_-=0\}$.
Since $2g_i'(t) \notin \pi \Z$ for any $t$, $x$ must be the constant map to $(0,\ldots,0)$.
By (c) in $G$-(5),
$X_{G_i}(q,p)= 2A_i(p_+\partial_{q_+} - q_+ \partial_{p_+}) + 2B_i p_-\partial_{q_-} - 2C_i q_- \partial_{p_-}$
on some neighborhood of $(0,\ldots,0)$.
Then, $\ind x>m$ follows from $A_i>\frac{(m+k)\pi}{2}$, $B_i>0$ and $C_i<0$.
\end{proof}

By $G$-(1), $D_-$ is an invariant set of $X_{G_i}$. Hence
$\mca{C}(G_i)$ is divided into two subsets:
\[
\mca{C}_-(G_i) = \bigl\{  x \in \mca{C}(G_i) \bigm{|} x\bigl([0,1]\bigr) \subset D_- \bigr\}, \quad
\mca{C}_+(G_i) = \bigl\{  x \in \mca{C}(G_i) \bigm{|} x\bigl([0,1]\bigr) \cap D_- = \emptyset \bigr\}.
\]
By $G$-(1), $\mca{C}_-(G_i)$ can be identified with $\mca{C}(F_i)$.

\begin{lem}\label{lem:C+}
If $x \in \mca{C}_+(G_i)$, then $x\bigl([0,1]\bigr) \subset C_{\delta_0} \cup A^+_{\delta_0} \times [1,\infty)$.
\end{lem}
\begin{proof}
Assume that there exists $\tau \in [0,1]$ such that $x(\tau) \notin C_{\delta_0} \cup A^+_{\delta_0} \times (1,\infty)$, hence
$x(\tau) \in \bigl(\partial D_- \setminus \psi_1( A^-_{\delta_0})\bigr) \times [1,\infty)$.
Let $I$ be the largest closed interval which contains $\tau$ and $x(I) \subset \bigl(\partial D_- \setminus \psi_1(A^-_{\delta_i})\bigr) \times [1,\infty)$.
Then $|I|>0$, and $x(\partial I)$ is contained in 
$\bigl(\psi_1(B_{\delta_i}) \cup (\partial D_- \cap N)\bigr) \times [1,\infty)$.

By $G$-(4), $X_{G_i}=\alpha_i \cdot (R,0)$ on $\bigl(\partial D_- \setminus \psi_1(A^-_{\delta_i})\bigr) \times [1,\infty)$.
Define $y:I \to \partial D_-$ by $y=\pi \circ x$, where $\pi$ is the projection to $\partial D_- \setminus \psi_1(A^-_{\delta_i})$.
Then $\dot{y}=\alpha_i R(y)$, $y(\partial I) \subset \psi_1(B_{\delta_i}) \cup (\partial D_- \cap N)$ and
$y(\tau) \notin \psi_1(A^-_{\delta_0})$.
Since $\delta_i < \delta(\alpha_i)$, $y(\partial I) \cap \psi_1(B_{\delta_i})=\emptyset$. Hence
$y(\partial I) \subset \partial D_- \cap N$ and $I=[0,1]$, 
but this contradicts $\alpha_i \notin \mca{A}(\partial D_-, \lambda_{a_0}, \partial D_- \cap N)$.
\end{proof}

By Lemma \ref{lem:G5} and Lemma \ref{lem:C+},
$\mca{C}_+(G_i)$ consists only of the constant map to $(0,\ldots,0)$ and its index is larger than $m$.
Hence $\WFC_{\le m} (G_i)$ is generated by elements of $\mca{C}_-(G_i)$.
On the other hand, since $\mca{C}_-(G_i)$ can be identified with $\mca{C}(F_i)$, 
there is an isomorphism of $\Z_2$ modules
$\WFC_{\le m}(F_i) \to \WFC_{\le m}(G_i)$.
By Lemma \ref{lem:annuluslemma}, 
if almost complex structures (which are used to define differential on $\WFC_*(F_i)$ and $\WFC_*(G_i)$)
satisfy assumption (2) in Lemma \ref{lem:annuluslemma} with $M^\inn=D_-$, 
this is an isomorphism of chain complexes.
Denote this isomorphism by $\Phi_i$.

Take $(i_k)_k$ as in $G$-(3), and consider following diagram:
\[
\xymatrix{
\WFC_{\le m}\bigl(F_{i_k}\bigr)\ar[r]\ar[d]_{\Phi_{i_k}}& \WFC_{\le m}(F_{i_{k+1}})\ar[d]^{\Phi_{i_{k+1}}}\\
\WFC_{\le m}\bigl(G_{i_k}\bigr)\ar[r]& \WFC_{\le m}(G_{i_{k+1}}).
}
\]
Horizontal arrows are monotone morphisms induced by monotone homotopies.

By $F$-(1) and $G$-(3), $F_{i_k}(x)<F_{i_{k+1}}(x)$ for any $x \in \hat{D_-}$, and 
$G_{i_k}(x)<G_{i_{k+1}}(x)$ for any $x \in \widehat{C_{\delta_0} \cup_{\psi_1} D_-}$.
Again by Lemma \ref{lem:annuluslemma}, 
if almost complex structures (which are used to define monotone morphisms)
satisfy assumption (2) in Lemma \ref{lem:annuluslemma} with $M^\inn=D_-$, the above diagram commutes.
Taking homology of this diagram and letting $i \to \infty$, 
we get (last equality follows from $G$-(2))
\[
\lim_{i \to \infty} \WFH_{\le m}(F_i) \cong \lim_{i \to \infty} \WFH_{\le m}(G_i) = \WFH_{\le m}\bigl(C_\delta \cup_{\psi_1} D_-, \omega_1, L_1\bigr)
\]
Hence we have proved (\ref{eq:goal}).

\end{document}